\documentclass[11pt ,a4paper]{amsart}
\usepackage[T1,T2A]{fontenc}
\usepackage[utf8]{inputenc}
\usepackage[a4paper,margin=1in]{geometry}
\usepackage{fancyhdr}
\usepackage{graphicx}
\usepackage{tikz-cd}
\usepackage{mathrsfs}
\usepackage{amsthm}
\usepackage{amssymb}
\usepackage{amsmath}
\usepackage{mathtools}
\usepackage{bm}
\usepackage{todonotes}
\usepackage{adjustbox}
\usepackage{hyperref}
\hypersetup{colorlinks=true, citecolor=black, linkcolor=black}
\usepackage[capitalize, nameinlink]{cleveref}
\usepackage{pdftexcmds}
\usepackage{bbm} 
\usepackage[bbgreekl]{mathbbol}
\usepackage[toc,page]{appendix}
\usepackage[foot]{amsaddr}
\usepackage{enumitem}
\usepackage{xcolor}
\DeclareSymbolFontAlphabet{\mathbb}{AMSb}
\DeclareSymbolFontAlphabet{\mathbbl}{bbold}
\DeclareMathSymbol\bbDelta  \mathord{bbold}{"01}
\usepackage[backend=biber, style=alphabetic, maxnames=10, maxalphanames=3, minalphanames=3]{biblatex}
\makeatletter
\renewcommand{\subsection}{%
  \@startsection{subsection}{2}{\z@}%
  {-3.25ex \@plus -1ex \@minus -.2ex}%
  {-1em}%
  {\normalfont\large\itshape}}
\makeatother

\usetikzlibrary{arrows}
\theoremstyle{plain}

\makeatletter
\newcommand{\colim@}[2]{%
  \vtop{\m@th\ialign{##\cr
    \hfil$#1\operator@font colim$\hfil\cr
    \noalign{\nointerlineskip\kern1.5\ex@}#2\cr
    \noalign{\nointerlineskip\kern-\ex@}\cr}}%
}
\newcommand{\colim}{%
  \mathop{\mathpalette\colim@{\rightarrowfill@\textstyle}}\nmlimits@
}
\makeatother

\newtheorem{mthm}{Theorem}[section]
\newtheorem{mcor}[mthm]{Corollary}
\newtheorem{mlem}[mthm]{Lemma}
\newtheorem{mprop}[mthm]{Proposition}
\newtheorem{mques}[mthm]{Question}
\newtheorem*{mthm*}{Theorem}
\theoremstyle{definition}

\theoremstyle{remark}

\theoremstyle{definition}
\newtheorem{mdef}[mthm]{Definition}
\newtheorem{mcon}[mthm]{Construction}

\theoremstyle{remark}
\newtheorem*{mnot}{Notation}
\title{Collisions of Nilpotent Higgs Fields}
\date{\today.}
\author{Marius Kjærsgaard}

\renewcommand{\varphi}{\phi}

\renewcommand{\epsilon}{\varepsilon}

\begin{document}
\begin{abstract}
    We prove the non-characteristicness of the completed self-sum of the universal nilpotent cone in $T^*\mathrm{Bun}_G$ relative to the moduli of curves, for reductive groups $G$ of semi-simple rank $1$. In particular, this implies the non-characteristicness of the $\mathcal{H}om$-sheaf between Betti automorphic sheaves for groups of semi-simple rank $1$. We also construct an example of a collision of universal nilpotent Higgs fields, proving that the completed self-sum is not non-characteristic for $\mathrm{GL}_n$, $n\ge 3$, relative to the moduli of curves of genus $g\ge 4$. 
\end{abstract}
\maketitle
\setcounter{section}{0}

\section{Introduction}
Fix a reductive group $G$ over an algebraically closed field $k$ of characteristic $0$. Let $\text{Bun}_G$ be the moduli stack of $G$-bundles on the universal curve over the moduli stack $\mathcal{M}_g$ of genus $g$ curves. That is the moduli problem classifying a smooth proper family of genus $g$ curves $X\to S$ and a $G$-bundle on $X$.

Inside the relative cotangent stack $T_{\mathcal{M}_g}^*\mathrm{Bun}_G$, there is a closed conic Lagrangian $\mathrm{Nilp}_\mathrm{rel}$ classifying nilpotent Higgs fields. The universal nilpotent cone $\mathrm{Nilp}_u\subset T^*\text{Bun}_G$ is the unique closed conic Lagrangian lift of the relative nilpotent cone to the absolute cotangent space \cite{nadler2026globalnilpotentconeuniversal}. This object is central to the study of the Betti automorphic category under deformations of curves, see for example \cite{nadler2023automorphicgluingfunctorbetti} or \cite{færgeman2026localconstancybettiautomorphic}.

The aim of this text is to prove the following geometric property of $\mathrm{Nilp}_u$.
\begin{mprop}\label{non-characteristicness}
For $G$ a reductive group of semi-simple rank $1$, the completed sum of $\mathrm{Nilp}_u$ with itself
\[\mathrm{Nilp}_u\widehat{+} \mathrm{Nilp}_u\subset T^*\mathrm{Bun}_G\] 
is non-characteristic with respect to $\mathcal{M}_g$.
\end{mprop}
Here $\widehat{+}$ denotes the completed sum of cones as defined by Kashiwara and Schapira, see \cite[Definition 6.2.3]{kashiwara1990sheaves} or Section 3 for a definition in the context of Artin stacks. Our main interest in this statement is the following consequence for automorphic sheaves.

\begin{mcor}\label{automorphic consequence}
    For $k=\mathbb{C}$, $G$ a reductive group of semi-simple rank 1 and any pair of automorphic sheaves $\mathcal{F},\mathcal{G}\in \mathrm{Sh}_{{\mathrm{Nilp}_u}}(\mathrm{Bun}_G)$, the sheaves $\mathcal{H}om(\mathcal{F},\mathcal{G})$ and $\mathcal{F}\otimes \mathcal{G}$
    are non-characteristic relative to $\mathcal{M}_g$.
\end{mcor}

This corollary can be used to give a simpler proof of Nadler and Yun's local constancy conjecture, see \cite[Conjecture 4.4.3]{nadler2023automorphicgluingfunctorbetti}, in the case of semi-simple rank 1 groups. However, Proposition \ref{non-characteristicness} turns out to be a semi-simple rank 1 miracle.

We construct an example proving that Proposition \ref{non-characteristicness} is really a special property for semi-simple rank $1$ groups.

\begin{mprop}
    For $G=\mathrm{GL}_n$, with $n\ge 3$, the completed sum $\mathrm{Nilp}_u\widehat{+}\mathrm{Nilp}_u$ is not non-characteristic with respect to $\mathcal{M}_g$ for $g\ge 4$.
\end{mprop}

The recent proof of Nadler and Yun's conjecture, see \cite{færgeman2026localconstancybettiautomorphic}, proves that $\Psi_*\mathcal{H}om(\mathcal{F},\mathcal{G})$ is a locally constant sheaf on $\mathcal{M}_g$, where $\Psi:\mathrm{Bun}_G\to \mathcal{M}_g$. However it does not establish Corollary \ref{automorphic consequence}. We do not know if Corollary \ref{automorphic consequence} is true for a general reductive group. 

\begin{mques}
    For any reductive group $G$ and any pair of automorphic sheaves $\mathcal{F},\mathcal{G}\in \mathrm{Sh}_{{\mathrm{Nilp}_u}}(\mathrm{Bun}_G)$, are the sheaves $\mathcal{H}om(\mathcal{F},\mathcal{G})$ and $\mathcal{F}\otimes \mathcal{G}$ non-characteristic relative to $\mathcal{M}_g$?
\end{mques}

\subsection{Lie algebra analogue}To motivate the argument for Proposition 1.1, let us first ignore the curves and consider the analogous construction for Lie algebras. Let $\mathfrak{g}$ be a Lie algebra, and let $N\subset \mathfrak{g}^\vee$ denote the nilpotent cone. For any $n\in N$, there exists a Borel subalgebra $\mathfrak{b}\subset \mathfrak{g}$ such that $n$ lies in the kernel of the natural map
\[
\mathfrak{g}^\vee\longrightarrow \mathfrak{b}^\vee.
\]
Given two elements $n,n'\in N$, we can choose potentially different Borel subalgebras $\mathfrak{b}$ and $\mathfrak{b}'$ that kill $n$ and $n'$, respectively. Since the intersection of any two Borel subalgebras contains a maximal torus, there exists a torus $\mathfrak{t}\subset \mathfrak{b}\cap\mathfrak{b}'$ such that $n+n'$ lies in the kernel of
\[
\mathfrak{g}^\vee\longrightarrow \mathfrak{t}^\vee.
\]
Thus, although the sum of two nilpotent elements need not be killed by a Borel subalgebra, it is always killed by a maximal torus.

In the context of Lie algebras, this statement does not give much information about the sum of the nilpotent cone with itself. For example, when $\mathfrak{g}=\mathfrak{sl}_2$, every element of $\mathfrak{sl}_2^\vee$ is killed by some maximal torus, reflecting the fact that every element of $\mathfrak{sl}_2^\vee$ is a sum of two nilpotent elements. However, once we introduce the curves back in the story, the analogous statement imposes a non-trivial condition on cotangent vectors arising from deformations of the curve. These are precisely the cotangent directions relevant to proving non-characteristicness.

\subsection{Collisions of absolute Higgs fields} 

The absolute cotangent space $T^*\mathrm{Bun}_G$ admits a moduli interpretation in terms of Atiyah bundles. Given a $T$-point of $\mathrm{Bun}_G$, represented by a $G$-bundle $P_G$ on a family of curves $X/T$, the cotangent space is identified with the space of global sections
\[
H^0(X,\mathrm{At}(P_G)^\vee\otimes\Omega_{X/T}),
\]
where $\mathrm{At}(P_G)^\vee$ denotes the dual Atiyah bundle associated to $P_G$ on $X/T$. We refer to such global sections $v$ as \textit{absolute Higgs fields}.

Similar to the case of Lie algebras, we say that an absolute Higgs field $v$ is \textit{nilpotent} if there exists a $B$-reduction $P_B$ of $P_G$ such that $v$ lies in the kernel of the natural map
\[
\mathrm{At}(P_G)^\vee\otimes\Omega_{X/T}
\longrightarrow
\mathrm{At}(P_B)^\vee\otimes\Omega_{X/T}.
\]
In the global case, however, we are not just interested in pointwise sums of nilpotent absolute Higgs fields. Instead, we consider so-called \textit{collisions of nilpotent absolute Higgs fields}. A collision of nilpotent absolute Higgs fields over a discrete valuation ring $R$ consists of a quadruple $(f,X,P_G,v)$, where $f\in\mathfrak{m}_R\setminus \{0\}$, $T_f=\operatorname{Spec}R[x]/(x(x-f))$, $X$ is a smooth proper curve over $T_f$, $P_G$ is a $G$-bundle on $X$, and
\[
v\in H^0(X,\mathrm{At}(P_G)^\vee\otimes\Omega_{X/T_f})
\]
is an absolute Higgs field such that the restriction of $v$ to each of the two irreducible components of $X$ is nilpotent.

\subsection{Idea of the proof}
Similarly to the case of Lie algebras, the nilpotency of $v$ on the generic fiber of each branch of $X$ implies the existence of two rational $B$-reductions of $P_G$. These reductions define quotients of $\mathrm{At}(P)^\vee \otimes \Omega_{X/T_f}$ that annihilate $v$ on their respective branches. Importantly, $\Omega_{X/T_f}^{\otimes 2}$ is a subbundle of both of these quotients.

We would then like to construct a further quotient vector bundle $D$ that factors through both of these quotients, such that $\Omega_{X/T_f}^{\otimes 2}$ is still a subbundle. However, the ambiguity in choosing a torus inside the intersection of the two Borels prevents us from defining such a quotient $D$ directly. Instead, we construct a quotient bundle $D_\mu$ of the dual Atiyah bundle over a fiberwise dense open subset of a formal neighborhood of a section $\mu$ of the family of curves. This formal quotient will still allow us to detect whether a collision of nilpotent absolute Higgs fields gives rise to a characteristic cotangent vector coming from $\mathcal{M}_g$.

By carefully choosing $\mu$ and a uniformization of the the data involved in the collision, the problem of constructing $D_\mu$ becomes independent of the family of curves $X$. We then collapse the formal branches of the completion of $X$ along $\mu$ to essentially reduce to a situation over $\mathrm{Spec }R[[z]]$. It turns out that in this formal setting the statement is true for $G$ of semi-simple rank 1, but false in general.

In fact, for $\mathrm{GL}_3$, considering formal collisions of Higgs fields leads to an example in which a collision of nilpotent absolute Higgs fields gives rise to characteristic cotangent vectors. In the last section, we construct this example explicitly for $\mathrm{GL}_3$ over a family of curves of genus $4$.

\subsection{Acknowledgements} I would like to thank Dustin Clausen for many discussions of mathematics in general and Joakim Færgeman for his collaboration on closely related topics.

\section{Dual Atiyah Bundles and The Universal Nilpotent Cone}
In this section we set up the basic properties of dual Atiyah bundles. Throughout this section $G$ is any affine algebraic group scheme over $\operatorname{Spec } k$, not necessarily reductive. Let $T$ be a base scheme over $\operatorname{Spec} k$, $X\to T$ be a scheme over $T$ and $P_G$ be a $G$-bundle on $X$.

Since the map $p:P_G\to X$ is smooth, we obtain a short exact sequence of quasi-coherent sheaves 
\[0\to p^* \Omega_{X/T}\to \Omega_{P_G/T}\to \Omega_{P_G/X}\to 0\]
on $P$. The functorial properties of Kähler differentials and the product formula provide descent data for this sequence along the map $P_G\to [P/G]\simeq X$. Let
\[0\to \Omega_{X/T}\to \mathrm{At}(P_G)^\vee \to \mathfrak{g}_{P_G}^\vee\to 0\]
denote the resulting sequence of quasi-coherent sheaves on $X$.

\begin{mdef}
    The \textit{dual Atiyah bundle} associated to $P_G$ on $X$ relative to $T$ is the quasi-coherent sheaf $\mathrm{At}(P_G)^\vee$ in the above sequence.
\end{mdef}

Note that in this generality the dual Atiyah bundle is not a vector bundle since $\Omega_{X/T}$ is not. As the notation suggests, the sheaf $\mathfrak{g}_{P_G}^\vee$ is isomorphic to the coadjoint bundle associated to $P_G$.

\begin{mlem}
    The quasi-coherent sheaf $\mathfrak{g}_P^\vee$ is isomorphic to the dual of the adjoint bundle associated to $P$.
\end{mlem}

\begin{proof}
Let $U\to X$ be an open on which $P_G$ is trivializable. Choosing a trivialization $P_G|_U\simeq G\times U$ provides an isomorphism of $\mathfrak{g}_{P_G}^\vee$ with the vector bundle of $G$-invariant differential forms on $G\times U$. This is precisely $\mathfrak{g}^\vee\times U$. Changing the trivialization of $P_G$ over $U$ acts by pulling back $G$-invariant relative differential forms, which is the same as the dual of the adjoint action on $\mathfrak{g}^\vee\times U$. Hence the gluing data for $\mathfrak{g}_{P_G}^\vee$ is exactly the gluing data for the dual adjoint bundle associated to $P_G$.
\end{proof}

We will only apply the definition of the dual Atiyah bundle directly in the case where $X/T$ is a smooth proper curve, in which case $\Omega_{X/T}$ is a vector bundle and hence the dual Atiyah bundle is also a vector bundle. However, we will also want an analogue of the dual Atiyah bundle for $G$-bundles on a formal neighborhood of a section $\mu:T\to X$. In this case the algebraic dual Atiyah bundle would not be a vector bundle or even a coherent sheaf. Instead we define the completed dual Atiyah bundle in the following setting. Let $T=\mathrm{Spec }R$ be the spectrum of a local ring, $X/T$ be a smooth curve and $\mu:T\to X$ be a section.  Let $\Gamma_\mu$ denote the corresponding closed subscheme and $I$ be the ideal sheaf on $X$ defining $\Gamma_\mu$. Consider the diagram
\[V(I)\to V(I^2)\to V(I^3)\to \cdots\]
Since $V(I)=T$ is an affine scheme, each scheme in the above diagram is a nilpotent extension of an affine scheme and hence affine. Let $X\widehat{_\mu}$ denote the colimit of the above diagram in the category of affine schemes. Let $X\widehat{_\mu}=\mathrm{Spec } R'$ and equip $R'$ with the inverse limit topology. Let $i_\mu:X\widehat{_\mu}\to X$ denote the natural map. Since $T$ is the spectrum of a local ring, we may choose an open neighborhood $U$ of $\mu$ such that $U\to T$ factors through an étale morphism $U\to \mathbb{A}^1_T$. By the infinitesimal lifting criterion for étale morphisms, choosing such an étale morphism provides an isomorphism $X\widehat{_\mu}\simeq \mathrm{Spec }R[[z]]$ induced by an isomorphism of topological rings $R'\cong R[[z]]$.

\begin{mdef}
    The \textit{completed Kähler differentials} $\Omega_{X\widehat{_\mu}/T}^c$ is the quasi-coherent sheaf on $X\widehat{_\mu}$ given by the Hausdorff completion of the algebraic Kähler differentials $\Omega_{R'/R}$ as a topological $R'$-module, see \cite[Section 20.4]{EGAIV4}. For any open subscheme $U\subseteq X\widehat{_\mu}$, let $\Omega_{U/T}^c$ denote the pullback of $\Omega_{X\widehat{_\mu}/T}^c$ to $U$.
\end{mdef}

The completed Kähler differentials form a line bundle for any $U\subseteq X\widehat{_\mu}$, since a choice of isomorphism of topological rings $R'\simeq R[[z]]$ induces an isomorphism on completed Kähler differentials. For $R[[z]]$ the complete Kähler differentials is the trivial line bundle $R[[z]]dz$.  There is a natural map of quasi-coherent sheaves $\Omega_{U/T}\to \Omega_{U/T}^\mathrm{c}$. Given a $G$-bundle $P_G$ over $U$, let
\[0\to \Omega_{U/T}^\mathrm{c}\to \mathrm{At}(P_G)^{\vee, c}\to \mathfrak{
g}_{P_G}^\vee \to 0\]
denote the exact sequence induced by the algebraic Atiyah sequence and the map $\Omega_{U/T}\to \Omega_{U/T}^\mathrm{c}$. 
\begin{mdef}
    For an open subscheme $U\subseteq X\widehat{_\mu}$ and a $G$-bundle $P_G$ over $U$, the \textit{completed dual Atiyah bundle} is the vector bundle $\mathrm{At}(P_G)^{\vee, c}$ on $U\subseteq X\widehat{_\mu}$.
\end{mdef}

We now have the following compatibility between the algebraic dual Atiyah bundle of a $G$-bundle on a curve and the completed dual Atiyah bundle of its restriction to a formal neighborhood of a section of the curve.

\begin{mprop}
    Let $T$ be the spectrum of a local ring, $X\to T$ a smooth curve, $\mu:T\to X$ a section and let $i_\mu:X\widehat{_\mu}\to X$ be the inclusion of the formal completion of $X$ at the graph of $\mu$. For any $G$-bundle $P_G$ on $X$, the map
    \[i_\mu^* \mathrm{At}(P_G)^\vee \to \mathrm{At}(i_\mu^* P_G)^{\vee,c}\]
    is an isomorphism.
\end{mprop}

\begin{proof} Consider the diagram
\[
\begin{tikzcd}
0 \arrow[r] & i_\mu^*\Omega_{X/T} \arrow[d] \arrow[r]        & i_\mu^*\mathrm{At}(P_G)^\vee \arrow[d] \arrow[r]  & i_\mu^* \mathfrak{g}_{P_G}^\vee \arrow[d, "\simeq"] \arrow[r]   & 0 \\
0 \arrow[r] & \Omega_{X\widehat{_\mu}/T} \arrow[r] \arrow[d] & \mathrm{At}(i_\mu^* P_G)^\vee \arrow[r] \arrow[d] & \mathfrak{g}_{i_\mu^* P_G}^\vee \arrow[r] \arrow[d, "="] & 0 \\
0 \arrow[r] & \Omega_{X\widehat{_\mu}/T}^c \arrow[r]         & {\mathrm{At}(i_\mu^* P_G)^{\vee,c}} \arrow[r]     & \mathfrak{g}_{i_\mu^* P_G}^\vee \arrow[r]                     & 0
\end{tikzcd}\]
Thus it suffices to prove that $i_\mu^* \Omega_{X/T}\to \Omega_{X\widehat{_\mu}/T}^c$ is an isomorphism. Since $T$ is the spectrum of a local ring, we may assume that $X$ admits an étale morphism $X\to \mathbb{A}^1_T$ over $T$. Then we have a diagram
\[
\begin{tikzcd}
X\widehat{_\mu} \arrow[d] \arrow[r, "\simeq"] & (\mathbb{A}^1_T)\widehat{_0} \arrow[d] \\
X \arrow[r]                                   & \mathbb{A}^1_T                        
\end{tikzcd}\]
Hence, since $X\to \mathbb{A}^1_T$ is étale and thus induces an isomorphism on Kähler differentials and the top map is induced by an isomorphism of topological rings, we are reduced to proving the statment for $(\mathbb{A}^1_T)\widehat{_0}\to \mathbb{A}^1_T$. In this case the statement follows as both modules of differential forms are free on the $1$-form $dz$. 
\end{proof}

The final property of dual Atiyah bundles that we need is the following proposition describing the automorphism of the dual Atiyah bundle induced by an automorphism of the trivial $\mathbb{G}_m$-bundle.

\begin{mprop}
    Let $U/T$ be a scheme, $f$ be an invertible function on $U$ and $\mathbb{G}_{m,U}\to \mathbb{G}_{m,U}$ the corresponding automorphism of the trivial $\mathbb{G}_m$-bundle over $U$. Then the induced map $\mathrm{At}(\mathbb{G}_{m,U})^\vee\overset{\sim}{\to} \mathrm{At}(\mathbb{G}_{m,U})^\vee$ is given by $\mathrm{At}(\mathbb{G}_{m,U})^\vee=\mathcal{O}_U\oplus \Omega_{U/T}\overset{\sim}{\to} \mathcal{O}_U\oplus \Omega_{U/T}=\mathrm{At}(\mathbb{G}_{m,U})^\vee$ where the middle map is given by the matrix
    \[\begin{pmatrix}
        1& \mathrm{dlog} f\\
        0& 1
    \end{pmatrix}.\]
\end{mprop}

\begin{proof}
    We immediately reduce to the case where $U$ is affine say $U=\mathrm{Spec } R'$, thus it suffices to prove the statement on global sections. The global sections of $\mathrm{At}(\mathbb{G}_{m,U})^\vee $ is the $R'$-module of global $\mathbb{G}_{m,U}$-invariant differential forms on $\mathbb{G}_{m,U}$ relative to $T$. This module is the direct sum of the differential forms pulled back from $U$ and multiples of the differential form $\mathrm{dlog} z$, where $z$ denotes the coordinate on $\mathbb{G}_m$. The automorphism induced by $f$ preserves the submodule of forms pulled back from $U$ and it sends $\mathrm{dlog} z$ to $\mathrm{dlog} fz=\mathrm{dlog} z + \mathrm{dlog} f$. Thus the map of algebraic dual Atiyah bundles is given by the matrix above.
\end{proof}

In the case where $U\subseteq X\widehat{_\mu}$ is an open subscheme, the analogous statement for completed dual Atiyah bundles is also true, where the exterior derivative is defined as the composition $\mathcal{O}_U\overset{d}{\to}\Omega_{U/T}\to \Omega_{U/T}^c$.

Given a smooth Artin stack $\mathcal{X}$, let $T^*\mathcal{X}$ denote the naive cotangent stack of $\mathcal{X}$. The relevance of Atiyah bundles to us comes from the following proposition.

\begin{mprop}
    The stack $T^*\mathrm{Bun}_G$ is the moduli space that to an affine test object $T$ assoicates the groupoid of a family of curves $X/T$, a $G$-bundle $P$ on $X$ and a global section $v\in H^0(X,\mathrm{At}(P)^\vee \otimes \Omega_{X/T})$.
\end{mprop}
\begin{proof}
    For an affine test scheme $T$, the $T$-points of the naive cotangent stack  $T^*\mathrm{Bun}_G$ over a given $T$-point $T\to \mathrm{Bun}_G$ is identified with the dual of the module of the first order deformations. The deformation problem associated to a $T$-point of $\mathrm{Bun}_G$ given by a pair of a curve $X/T$ and a $G$-bundle $P_G$ on $X$ is equivalent to the problem of deforming $X$ over $BG\times T$. Since $X\to BG \times T$ is smooth, the first order deformations of this problem is classified by $H^1(X,T_{X/BG\times T})$ and hence, by Serre duality and since $T$ is affine, the dual module is identified with $H^0(X,T_{X/BG\times T}^\vee \otimes \Omega_{X/T})$. We can identify $T_{X/BG\times T}^\vee$ by descending along the map $T\to BG\times T$ and see that it is exactly the quasi-coherent sheaf on $X$ obtained by descending $T_{P_G/T}^\vee=\Omega_{P_G/T}$ along the map $P_G\to X$. Hence $T_{X/BG\times T}^\vee$ is the dual Atiyah bundle $\mathrm{At}(P_G)^\vee$ and thus the $T$-points of $T^*\mathrm{Bun}_G$ over $(X/T,P_G)$ is given by the global sections of $H^0(X,\mathrm{At}(P_G)^\vee \otimes \Omega_{X/T})$.
\end{proof}

Motivated by the proposition above we make the following definition.
\begin{mdef}
    An \textit{absolute Higgs field} on a $G$-bundle $P_G$ over a curve $X/T$ is a global section $v\in H^0(X,\mathrm{At}(P)^\vee \otimes \Omega_{X/T})$. 
\end{mdef}

We end this section by defining the universal nilpotent cone inside $T^*\mathrm{Bun}_G$. Let
\[H_\mathrm{rel}:T^*_{\mathcal{M}_g}\mathrm{Bun}_G\to \mathrm{Hitch}_G\]
denote the relative Hitchin fibration. The relative nilpotent cone $\mathrm{Nilp}_\mathrm{rel}$ is the closed conical substack given by the $H_\mathrm{rel}^{-1}(\mathbf{0})$, where $\mathbf{0}$ denotes the zero-section of the Hitchin base.

\begin{mthm}[\cite{nadler2023automorphicgluingfunctorbetti}]
    There exist a unique closed, conical, Lagrangian subset $\mathrm{Nilp}_u\subset T^*\mathrm{Bun}_G$ which maps bijectively to $\mathrm{Nilp}_\mathrm{rel}$ under the projection $T^*\mathrm{Bun}_G\to T_{\mathcal{M}_g}^*\mathrm{Bun}_G$.
\end{mthm}

From Nadler and Yun's construction of $\mathrm{Nilp}_u$, it follows that the field valued points of $\mathrm{Nilp}_u$ can be described as follows: A $k$-point of $T^*\mathrm{Bun}_G$ given by an absolute Higgs field $v$ on a $G$-bundle $P_G$ on a curve $X/\operatorname{Spec }k$ defines a $k$-point of $\mathrm{Nilp}_u$ if and only if there exist a $B$-reduction of $P_G$ such that $v$ is in the kernel of the map
\[
H^0(X,\mathrm{At}(P_G)^\vee\otimes\Omega_{X/T})
\longrightarrow
H^0(X,\mathrm{At}(P_B)^\vee\otimes\Omega_{X/T}).
\]

\section{The Completed Sum of Cones for Smooth Artin Stacks}
Let $\mathcal{X}$ be a smooth Artin stack and $\mathcal{N},\mathcal{N}'\subseteq T^*\mathcal{X}$ be closed conical subsets of the cotangent space of $\mathcal{X}$. Here $T^*\mathcal{X}$ denotes the naive cotangent stack, in particular beware that $T^*\mathcal{X}$ is in general not a vector bundle over $\mathcal{X}$ but only the affine cone associated to a coherent sheaf. In this section we will generalize the \textit{completed sum} $\mathcal{N}\widehat{+}\mathcal{N}'\subseteq T^*\mathcal{X}$, as defined by Kashiwara and Schapira in the case of smooth manifolds, to the context of smooth Artin stacks. 

We start by recalling the definition of the completed sum for $X$ a smooth variety. In this case the symplectic structure on $T^*X$ defines an isomorphism $T^*T^*X\simeq TT^*X$ of vector bundles over $T^*X$. The map on cotangent spaces induced by $T^*X\to X$ gives a map $T^*X\times_X T^*X\to T^*T^*X$ over $T^*X$, where $T^*X\times_X T^*X\to T^*X$ is given by projection to the first factor. 
\begin{mdef}
    The Hamiltonian embedding $H_X:T^*X\hookrightarrow TT^*X$ is the closed immersion over $X$ given by the composition $T^*X\to T^*X\times_X T^*X\to  T^*T^*X\simeq TT^*X$, where the first map is given $(0,\mathrm{id}_{T^*X})$.
\end{mdef}

This is the definition used in Kashiwara-Schapira, the definition that will be more useful to us is the moduli description of the Hamiltonian embedding. Consider the map $G_X:T^*X\to TT^*X$ defined by the following moduli interpretation. A $T$-valued point of $T^*X$ is a pair $(x,v)$, where $x:T\to X$ and $v\in H^0(T, x^* \Omega_{X}^1)$. Under the map $G_X$ the $T$-valued point $(x,v)$ is mapped to the $T$-valued point of $TT^*X$ given by the $T[\varepsilon]$-valued point of $T^*X$ given by $(x_\varepsilon,\varepsilon v)$, where $x_\varepsilon:T[\varepsilon]\to X$ is the map $T[\varepsilon]\to T\to X$ and 
\[\varepsilon v\in H^0(T,x^*\Omega_X^1)\oplus \varepsilon H^0(T,x^*\Omega_X^1)\simeq H^0(T,x_\varepsilon^*\Omega_X^1).\]
We will now prove that the map $G_X$ is in fact the Hamiltonian embedding.

\begin{mlem}
    For $X$ a smooth algebraic variety, the maps $G_X$ and $H_X$ are equal.
\end{mlem}
\begin{proof}
    Since the problem can be checked Zariski locally on $X$, we may assume that $X$ admits an étale morphism $g:X\to \mathbb{A}_k^n$. Then we have the following diagram
    \[
\begin{tikzcd}
T^*X \arrow[r, "H_X"]                                                                                                     & TT^*X                                                          \\
T^*\mathbb{A}^n_k\times_{\mathbb{A}^n_k} X \arrow[r, "H_{\mathbb{A}^n_k}\times_{0_{\mathbb{A}^n_k}} 0_X"] \arrow[u, "dg"] & T(T^*\mathbb{A}^n_k\times_{\mathbb{A}^n_k} X) \arrow[u, "Tdg"]
\end{tikzcd}\]
    Since $g$ is étale, the vertical maps are isomorphisms. To see that the diagram commutes, note that by construction the symplectic form on $T^*X$ is identified with the pullback of the standard symplectic form on $T^*\mathbb{A}^n_k$. The commutativity then follows since the composition $T^*\mathbb{A}^n_k\to TT^*\mathbb{A}^n_k \to T^*\mathbb{A}^n_k$ of the Hamiltonian embedding and the projection to $T^*\mathbb{A}^n_K$ is the zero map of vector bundles over $\mathbb{A}^n_k$. So the composition $T^*\mathbb{A}^n_k\times_{\mathbb{A}^n_k}X\to TT^*\mathbb{A}^n_k \times_{\mathbb{A}^n_k}X\simeq T(T^*\mathbb{A}^n_k \times_{\mathbb{A}^n_k}X)\simeq  TT^*\mathbb{A}^n_k\times_{T\mathbb{A}^n_k} TX$, where the first map is given by $H_{\mathbb{A}^n_k}\times \mathrm{id}_X$, is equal to $H_{\mathbb{A}^n_k}\times_{0_{\mathbb{A}^n_k}} 0_X$.

    It is easy to see from the moduli description of $G_X$ that the analogous diagram with $H_X$ replaced by $G_X$ also commutes, see the start of the proof of Lemma 3.4 below for an argument. Thus since the vertical maps are isomorphisms, it suffices to show that $H_{\mathbb{A}_k^n}$ is equal to $G_{\mathbb{A}_k^n}$.

    In the case $X=\mathbb{A}^n_k$, the Hamiltonian isomorphism is given by 
    \[T^*T^*\mathbb{A}^n=\mathbb{A}^n\times \mathbb{A}^{n,\vee}\times (\mathbb{A}^n\times \mathbb{A}^{n,\vee})^\vee\simeq \mathbb{A}^n\times \mathbb{A}^{n,\vee}\times \mathbb{A}^{n,\vee}\times \mathbb{A}^{n,\vee\vee}\simeq \mathbb{A}^n\times \mathbb{A}^{n,\vee}\times \mathbb{A}^n \times \mathbb{A}^{n,\vee}=TT^*\mathbb{A}^n\]
Thus the Hamiltonian embedding sends a point $(x,v)\in T^*\mathbb{A}^n$ to the point $(x,0,0,v)\in \mathbb{A}^n\times \mathbb{A}^{n,\vee}\times \mathbb{A}^n \times \mathbb{A}^{n,\vee}=TT^*\mathbb{A}^n$. On the other hand consider $(x,v)$ as the $k$-point on $T^*\mathbb{A}^n$ given by the map $x:k[x_1,\dots,x_n]\to k$ and the element $v$ in the $k$-vector space generated by $dx_1,\dots,dx_n$. By the moduli description of $G_X$, this is mapped to the $k[\varepsilon]$-point of $T^*\mathbb{A}^n$ given by $x:k[x_1,\dots,x_n]\to k\to k[\varepsilon]$ and the element $\varepsilon v$ in the free $k[\varepsilon]$-module generated by $dx_1,\dots,dx_n$. Modding out with $\varepsilon$ we see that the projection to the first two factors of $\mathbb{A}^n\times \mathbb{A}^{n,\vee}\times \mathbb{A}^n \times \mathbb{A}^{n,\vee}=TT^*\mathbb{A}^n$ is given by $(x,0)$. Looking at the coefficents of $\varepsilon$ it follows that the projection to the two last factors is given by $(0,v)$. Hence the two maps are equal for $\mathbb{A}^n$.
\end{proof}

Note that the moduli description of the Hamiltonian embedding extends to any smooth Artin stack $\mathcal{X}$. To be precise let $T^*\mathcal{X}$ denote the naive cotangent stack of $\mathcal{X}$, let $\Omega_\mathcal{X}^1$ denote the naive sheaf of differential 1-forms and let $TT^*\mathcal{X}$ denote the mapping stack $\mathrm{Map}(k[\varepsilon],T^*\mathcal{X})$. Then we have a morphism $T^*\mathcal{X}\to TT^*\mathcal{X}$ given by sending a $T$-point of $T^*\mathcal{X}$ given by a pair $(x,v)$, where $x:T\to \mathcal{X}$ and $v\in H^0(T,x^* \Omega_\mathcal{X}^1)$, to the $T$-point of $TT^*\mathcal{X}$ given by the $T[\varepsilon]$-point $(x_\varepsilon,\varepsilon v)$ of $T^*\mathcal{X}$, where $x_\varepsilon$ is the map $T[\varepsilon]\to T\to \mathcal{X}$ and
\[\varepsilon v\in H^0(T,x^*\Omega_\mathcal{X}^1)\oplus \varepsilon H^0(T,x^*\Omega_\mathcal{X}^1)\cong H^0(T[\varepsilon],x_\varepsilon^* \Omega^1_\mathcal{X}).\]

\begin{mdef}
    For $\mathcal{X}$ a smooth Artin stack, define the Hamiltonian embedding of $\mathcal{X}$ to be the morphism $H_\mathcal{X}:T^*\mathcal{X}\to TT^*\mathcal{X}$ described above.
\end{mdef}

Given a smooth morphism $f:\mathcal{Y}\to \mathcal{X}$ of smooth Artin stacks, we obtain a correspondence
\[T^*\mathcal{Y}\overset{df}{\hookleftarrow}T^*\mathcal{X}\times_{\mathcal{X}}\mathcal{Y}\overset{f_t}{\to} T^*\mathcal{X}\]
Where the map $df$ is a closed immersion. The Hamiltonian embedding has the following functoriality in smooth morphisms.

\begin{mlem}
    Let $f:\mathcal{Y}\to \mathcal{X}$ be a smooth morphism of smooth Artin stacks. The diagram
\[
\begin{tikzcd}
T^*\mathcal{Y} \arrow[r, "H_{\mathcal{Y}}"]                                                                                            & TT^*\mathcal{Y}                                                         \\
T^*\mathcal{X}\times_{\mathcal{X}}\mathcal{Y} \arrow[u, "df", hook] \arrow[r, "H_{\mathcal{X}}\times_{0_\mathcal{Y}} 0_{\mathcal{X}}"] & T(T^*\mathcal{X}\times_{\mathcal{X}}\mathcal{Y}) \arrow[u, "Tdf", hook]
\end{tikzcd}\]
    is Cartesian.
\end{mlem}

\begin{proof}
    To see that the diagram commutes, consider a $T$-point on $T^*\mathcal{X}\times_{\mathcal{X}}\mathcal{Y}$ given by $y:T\to \mathcal{Y}$ and $v\in H^0(T,x^*\Omega_{\mathcal{X}}^1)$, where $x:T\to \mathcal{X}$ is the composition of $y$ and $f$. Let $y$ and $v'\in H^0(T,y^*\Omega_\mathcal{Y}^1)$ denote the image of this $T$-point under $df$. Then $df$ followed by $H_\mathcal{Y}$ is given by the $T[\varepsilon]$-point of $T^*\mathcal{Y}$ given by $y_\varepsilon:T[\varepsilon]\to \mathcal{Y}$ and $\varepsilon v'\in H^0(T[\varepsilon],y_\varepsilon^*\Omega_\mathcal{Y}^1)$. Under the other composition the $T$-point is first mapped to the $T[\varepsilon]$-point of $T^*\mathcal{X}\times_{\mathcal{X}}\mathcal{Y}$ given by $y_\varepsilon$ and $\varepsilon v\in H^0(T[\varepsilon],x_\varepsilon^* \Omega_\mathcal{X}^1)$. Now the pullback map $H^0(T[\varepsilon],x_\varepsilon^*\Omega_\mathcal{X}^1)\to H^0(T[\varepsilon],y_\varepsilon^*\Omega_\mathcal{Y}^1)$ maps $\varepsilon v$ to $\varepsilon v'$ hence giving the commutativity.

    To prove that the diagram is Cartesian, it suffices to prove that a $T$-point of $T^*\mathcal{Y}$ factors through $df$ if and only if its image under $H_\mathcal{Y}$ factors through $Tdf$. The commutativity of the diagram implies that if a $T$-point factors through $df$ then its image factors through $Tdf$. Conversely, let $(y,v)$ be a $T$-point of $T^*\mathcal{Y}$, with $y:T\to \mathcal{Y}$ and $v\in H^0(T,y^*\Omega^1_\mathcal{Y})$, and suppose that the image under the Hamiltonian embedding factors through $Tdf$. Hence the $T[\varepsilon]$-point $(y_\varepsilon,\varepsilon v)$ of $T^*\mathcal{Y}$ factors through $df$. Let $x$ denote the image of $y$ under $f$. Then there is an element $w\in H^0(T[\varepsilon],x_\varepsilon^* \Omega_\mathcal{X}^1)$ which maps to $\varepsilon v$. But then $w$ must be of the from $w=\varepsilon w'$ for $w'\in H^0(T,x^*\Omega_\mathcal{X}^1)$ and $w'$ maps to $v$. Hence $(y,v)$ is in the image of $df$, being the image of $(y,w')$. 
\end{proof}

For $\mathcal{X}$ a stack, let $D_\mathcal{X}$ denote the moduli problem sending a test object $T$ to the space of pairs $(f,x)$ where $f$ is a function $f:T\to \mathbb{A}^1$ and $x$ is a map $T_f:=\underline{\mathrm{Spec}} \mathcal{O}_T[x]/x(x-f)\to \mathcal{
X}$. Note that $D_\mathcal{X}$ admits a map $F:D_\mathcal{X}\to \mathbb{A}^1$ given by sending $(f,x)$ to $f$. The fiber over $0\in \mathbb{A}^1$ is equivalent to the tangent stack $T\mathcal{X}$. Composing with the Hamiltonian embedding gives a map $H'_X:T^*X\to D_{T^*X}$. Let $O_\mathcal{X}\hookrightarrow D_\mathcal{X}$ denote the open substack given by the preimage of $\mathbb{A}^1\setminus \{0\}$. When $\mathcal{X}$ is a smooth variety, $D_\mathcal{X}$ identifies with the deformation to normal cone of the diagonal $\mathcal{X}\to \mathcal{X}\times \mathcal{X}$.

Given a test scheme $T$ and $f:T\to \mathbb{A}^1$, let $T_{f,+}$ denote the closed subscheme of $T_f$ given by $V(x)$ and similarly let $T_{f,-}$ denote the closed subscheme given by $V(x-f)$. As schemes these are both isomorphic to $T$. Thus we obtain a map $P_\mathcal{X}:D_\mathcal{X}\to \mathcal{X}\times \mathcal{X}$ given by restricting $x$ to $T_{f,+}$ and $T_{f,-}$. From the moduli description we see that $P_\mathcal{X}$ and $F:D_\mathcal{X}\to \mathbb{A}^1$ induce an equivalence from $O_\mathcal{X}$ to $\mathcal{X}\times \mathcal{X}\times \mathbb{A}\setminus\{0\}$.

\begin{mdef}
    Let $X$ be a smooth variety and $\mathcal{N},\mathcal{N}'\subseteq T^*X$ be closed conical subsets. The \textit{completed sum} $\mathcal{N}\widehat{+}\mathcal{N}'$ is the closed conical subset of $T^*X$ given by $H'^{-1}_X(\overline{O_{T^*X}\cap P_{T^*X}^{-1}(\mathcal{N}\times \mathcal{N}')})$.
\end{mdef}

We will now extend this definition to smooth Artin stacks. To specify a closed subset of $T^*\mathcal{X}$ is equivalent to specifying a compatible system of closed subsets of $T^*\mathcal{X}\times_{\mathcal{X}} Y$ for all smooth morphisms $f:Y\to \mathcal{X}$ with $Y$ a scheme. The following lemma gives this compatibility for defining $\widehat{+}$.

\begin{mlem}
    Let $f:Y\to X$ be a smooth morphism of smooth varieties and $\mathcal{N},\mathcal{N}'\subseteq T^*X$ be closed conical subsets, there is an equality of closed subsets of $T^*Y$
    \[df(f_t^{-1}(\mathcal{N}))\widehat{+} df(f_t^{-1}(\mathcal{N}'))=df(f_t^{-1}(\mathcal{N}\widehat{+}\mathcal{N}')).\]
\end{mlem}
\begin{proof}
    Consider the following commutative diagram
    \[
\begin{tikzcd}
T^*Y \arrow[d, "H_Y", hook] & T^*X\times_{X}Y \arrow[d, "H_X\times_{0_X} 0_Y", hook] \arrow[r,"f_t"] \arrow[l, "df"', hook'] & T^*X \arrow[d, "H_X", hook] \\
TT^*Y \arrow[d, hook]       & T(T^*X\times_{X}Y) \arrow[r] \arrow[l, hook'] \arrow[d, hook]                            & TT^*X \arrow[d, hook]       \\
D_{T^*Y}                    & D_{T^*X\times_{X}Y} \arrow[l, hook'] \arrow[r]                                           & D_{T^*X}                   
\end{tikzcd}\]
Since $O_{T^*X}$ maps isomorphically to $T^*X\times T^*X\times \mathbb{A}^1\setminus \{0\}$, transporting $O_{T^*X}\cap P_{T^*X}^{-1}(\mathcal{N}\times \mathcal{N}')$ along the bottom correspondance gives exactly $O_{T^*Y}\cap P_{T^*Y}^{-1}(df(f_t^{-1}(\mathcal{N}))\times df(f_t^{-1}(\mathcal{N}')))$. Since the map $T^*X\times_X Y\to T^* X$ is smooth so is the map $D_{T^*X\times_X Y}\to D_{T^*X}$, in particular the latter map is open. Thus transporting along the bottom correspondence commutes with taking closures. The result now follows since Lemma 3.4 implies that transporting and pulling back along the Hamiltonian embedding is the same as pulling back along the Hamiltonian embedding and then transporting.
\end{proof}

For any smooth map $f:U\to \mathcal{X}$ with $U$ a scheme, it follows from Lemma 3.4 that the closed set $df(f_t^{-1}(\mathcal{N}))\widehat{+}df(f_t^{-1}(\mathcal{N}'))\subset T^*U$ is a closed subset of  $T^*\mathcal{X}\times_\mathcal{X} U$. Thus we can make the following definition.

\begin{mdef}
    Let $\mathcal{X}$ be a smooth Artin stack and $\mathcal{N},\mathcal{N'}\subseteq T^*\mathcal{X}$ be closed conical subsets. The \textit{completed sum} $\mathcal{N}\widehat{+}\mathcal{N}'\subseteq T^*\mathcal{X}$ is the closed conical subset which, for each smooth map $f:U\to \mathcal{X}$ with $U$ a scheme, is given by
    \[(\mathcal{N}\widehat{+}\mathcal{N}')|_U:=df(f_t^{-1}(\mathcal{N}))\widehat{+}df(f_t^{-1}(\mathcal{N}'))\subseteq T^*\mathcal{X}\times_\mathcal{X} U.\]
\end{mdef}

Suppose we are given a smooth map $g:U'\to U$, to see that the closed subsets in Definition 3.7 form a compatible system consider the following diagram

\[
\begin{tikzcd}
T^*U'                               &                                                                                               &                \\
T^*U\times_U U' \arrow[u, hook'] \arrow[d] & T^*\mathcal{X}\times_{\mathcal{X}} U' \arrow[rd] \arrow[lu, hook'] \arrow[l, hook'] \arrow[d] &                \\
T^*U                                & T^*\mathcal{X}\times_{\mathcal{X}} U \arrow[r] \arrow[l, hook']                               & T^*\mathcal{X}
\end{tikzcd}\]
By applying Lemma 3.6 to the map $g$ and using that the square in the above diagram is Cartesian, it follows that the preimage of $(\mathcal{N}\widehat{+}\mathcal{N}')|_U\subset T^*\mathcal{X}\times_\mathcal{X} U$ under $T^*\mathcal{X}\times_\mathcal{X} U'\to T^*\mathcal{X}\times_\mathcal{X} U$ is $(\mathcal{N}\widehat{+}\mathcal{N}')|_{U'}$.

\section{Collisions of Nilpotent Higgs Fields}

In this section we will reinterpret the non-characteristicness of  $\mathrm{Nilp}_u\widehat{+}\mathrm{Nilp}_u\subseteq T^* \mathrm{Bun}_G$ over $\mathcal{M}_g$ in terms of \textit{collisions of nilpotent absolute Higgs fields}.

\subsection{The setup} Throughout we will be working with several different geometric objects living over $T=\mathrm{Spec } R$ and $T_f:=\mathrm{Spec } R[x]/x(x-f)$ for $R$ a discrete valuation ring and $f\in \mathfrak{m}\setminus\{0\}$, we therefore introduce the following general notation.
\begin{mnot}
    Let $R$ be a discrete valuation ring with maximal ideal $\mathfrak{m}$ and $f\in \mathfrak{m}\setminus\{0\}$. For any geometric object $S$ over $T_f:=\operatorname{Spec} R[x]/x(x-f)$, let $S_+$ denote the pullback of $S$ to the closed subscheme $V(x)\subset T_f$. Similarly let $S_-$ denote the pullback to $V(x-f)\subset T_f$. Let $\eta$ and $s$ denote the generic and the special fiber of $T=\operatorname{Spec} R$ respectively. For any geometric object $S$ over $T$ let $S_s$ and $S_\eta$ denote the pullback of $S$ to the special and the generic fiber of $T$ respectively. 
\end{mnot}

Since $\mathrm{Bun}_G$ is a smooth Artin stack of locally finite type, the non-characteristicness of $\mathrm{Nilp}_u \widehat{+}\mathrm{Nilp}_u\subseteq T^* \mathrm{Bun}_G$ relative to $\mathcal{M}_g$ is equivalent to the non-characteristicness of the closed conical subsets $df(f_t^{-1}(\mathrm{Nilp}_u))\widehat{+}df(f_t^{-1}(\mathrm{Nilp}_u))\subseteq T^*U$ relative to $\mathcal{M}_g$, for all smooth maps $f:U\to \mathrm{Bun}_G$ with $U$ a scheme of finite type over $\mathrm{Spec }k$. Given $f:U\to \mathrm{Bun}_G$ consider the following diagram
\[
\begin{tikzcd}
T^*U \arrow[d, "H'_U"] & T^*\mathrm{Bun}_G\times_{\mathrm{Bun}_G} U \arrow[r] \arrow[d, "H'_{\mathrm{Bun}_G}\times_{0_{\mathrm{Bun}_G}} 0_U"] \arrow[l, hook'] & T^*\mathrm{Bun}_G \arrow[d, "H'_{\mathrm{Bun}_G}"] \\
D_{T^*U}               & D_{T^*\mathrm{Bun}_G\times_{\mathrm{Bun}_G} U} \arrow[r] \arrow[l, hook']                                                             & D_{T^*\mathrm{Bun}_G}                             
\end{tikzcd}\]
As in the previous section 
\begin{align*}
    O_{T^*U}\cap P_{T^*U}^{-1}(df(f_t^{-1}(\mathrm{Nilp}_u))\times df(f_t^{-1}(\mathrm{Nilp}_u)))&=\\
    O_{T^*\mathrm{Bun_G}\times_{\mathrm{Bun_G}} U}\cap P_{T^*\mathrm{Bun_G}\times_{\mathrm{Bun_G}}U}^{-1}&(f_t^{-1}(\mathrm{Nilp}_u)\times f_t^{-1}(\mathrm{Nilp}_u))=: O_{{\mathrm{Nilp}_u},U}.
\end{align*}
Since $D_{T^*\mathrm{Bun_G}\times_{\mathrm{Bun_G}}U}$ is a representable by a finite type scheme over $\operatorname{Spec}k$, it follows that geometric points in the closure of a locally closed set $Z$ can be detected by maps from strictly Henselian discrete valuation rings. More precisely, given a geometric point $x\in \overline{Z}$ there exist a strictly Henselian discrete valuation ring $R$ and a map from $\operatorname{Spec } R$ such that the generic point factors through $Z$ and the closed point is $x$.

\subsection{Collisions} Let $R$ be a strictly Henselian discrete valuation ring with separably closed residue field $k$ and maximal ideal $\mathfrak{m}$. Let $T:=\operatorname{Spec} R\to D_{T^*U}$ be a map detecting a $k$-point in the closure of $O_{{\mathrm{Nilp}_u},U}$ with the $k$-point in the closed substack $TT^*U$. Since $O_{\mathrm{Nilp}_u,U}$ is a subset of $D_{T^*\mathrm{Bun_G}\times_{\mathrm{Bun_G}}U}$, it follows that the map $T\to D_{T^*U}$ factors through $D_{T^*\mathrm{Bun_G}\times_{\mathrm{Bun_G}}U}$ and the underlying $k$-point is in the closed substack $T(T^*\mathrm{Bun_G}\times_{\mathrm{Bun_G}}U)$. Consider the resulting map $T\to D_{T^*\mathrm{Bun}_G}$, unwinding the definitions such a map gives rise to the following data.

\begin{mdef}
    A \textit{collision of nilpotent absolute Higgs fields over} $T$ is a quadruple $(f,X,P,v)$, where $f\in \mathfrak{m}\setminus \{0\}$, $X$ is a smooth proper curve over $T_f:=\mathrm{Spec} R[x]/x(x-f)$, $P_G$ is a $G$-bundle on $X$, $v\in H^0(X,\mathrm{At}(P_G)^\vee\otimes \Omega^1_{X/T_f})$ such that $v_{+,\eta}$, respectively $v_{-,\eta}$, is a nilpotent absolute Higgs field on $P_{+,\eta}$, respectively $P_{-,\eta}$.
\end{mdef}

An absolute Higgs field $v$ on a $G$-bundle $P$ over a curve $X/T$ defines a cotangent vector pulled back from $\mathcal{M}_g$ exactly if $v$ is in the image of the map
 \[H^0(X,\Omega_{X/T}^{\otimes 2})\hookrightarrow H^0(X,\mathrm{At}(P)^\vee\otimes \Omega_{X/T}).\]
Thus the non-characteristicness of  $\mathrm{Nilp}_u\widehat{+}\mathrm{Nilp}_u\subseteq T^* \mathrm{Bun}_G$ over $\mathcal{M}_g$ is a consequence of the following statement.

\begin{mprop}
    Suppose $G$ has semi-simple rank $1$ and let $(f,X,P,v)$ be a collision of nilpotent absolute Higgs fields. If $v_s$ is in the image of the map
    \[H^0(X_s,\Omega_{X_s/(T_f)_s}^{\otimes 2} )\hookrightarrow H^0(X_s,\mathrm{At}(P_s)^\vee \otimes \Omega_{X_s/(T_f)_s}),\]
    then $v_s=0$.
\end{mprop}

To see that Proposition 4.2 implies Proposition 1.1, note that if $(X_0,P_0,v_0)$ is a $k$-point of $\mathrm{Nilp}_u\widehat{+}\mathrm{Nilp}_u$, then the image under the Hamiltonian embedding $(X_0[\varepsilon],P_0[\varepsilon],\varepsilon v_0)$ extends to a collision of nilpotent Higgs fields over the spectrum of a strictly Henselian discrete valuation ring $T$, with an isomorphism of $(X_0[\varepsilon],P_0[\varepsilon],\varepsilon v_0)$ with $(X_s,P_s,v_s)$. In particular if $v_0$ is pulledback from $\mathcal{M}_g$, then $v_s=\varepsilon v_0$ satisfies the assumption of Proposition 4.2 and hence $\varepsilon v_0=v_s=0$. Recall that
\[ H^0(X[\varepsilon],\Omega_{X[\varepsilon]/k[\varepsilon]}^1)\simeq H^0(X,\Omega_X^1)\oplus \varepsilon H^0(X,\Omega_X^1)\]
thus $v_0=0$. From now on we assume that $G$ has semi-simple rank $1$. We will prove Proposition 4.2 by proving the following proposition.

\begin{mprop}
    Let $(f,X,P,v)$ be a collision of nilpotent absolute Higgs fields over $T$. There exist a section $\mu: T_f\to X$, an open subscheme $U\subseteq X\widehat{_\mu}$, which is fiberwise dense over $T_f$, and a quotient vector bundle $D_\mu$ of $\mathrm{At}(P|_U)^{\vee,c}$ over $U$, such that the composition $\Omega_{U/T_f}^c\to \mathrm{At}(P|_U)^{\vee,c}\to D_\mu$ defines a subvector bundle of $D_\mu$ and $v$ is in the kernel of the map
    \[H^0(X,\mathrm{At}(P)^\vee\otimes \Omega_{X/T_f})\to H^0(U,D_\mu \otimes \Omega_{U/T_f}^c).\]
\end{mprop}

We will now see how Proposition 4.3 implies Proposition 4.2.

\begin{proof}[Proof of Proposition 4.2 assuming Proposition 4.3]
    Suppose we are in the setting of Proposition 4.2 and have chosen the data that exist by Proposition 4.3.
    Since $v$ is in the kernel of the map $H^0(X,\mathrm{At}(P)^\vee\otimes \Omega_{X/T_f})\to H^0(U,D_\mu \otimes \Omega_{U/T_f}^c)$, it follows that $v_s$ is in the kernel of
    \[H^0(X_s,\mathrm{At}(P_s)^\vee\otimes \Omega_{X_s/(T_f)_s})\to H^0(U_s,D_\mu|_{U_s} \otimes \Omega_{U_s/(T_f)_s}^c).\]
    By assumption $v_s$ is in the image of the map
    \[H^0(X_s,\Omega_{X_s/(T_f)_s}^{\otimes 2} )\hookrightarrow H^0(X_s,\mathrm{At}(P_s)^\vee \otimes \Omega_{X_s/(T_f)_s}).\]
    It follows that $v_s|_{U_s}$ is an element of $H^0(U_s,\Omega_{U_s/(T_f)_s}^{c,\otimes 2} )$ in the kernel of the map
    \[H^0(U_s,\Omega_{U_s/(T_f)_s}^{c,\otimes 2})\to H^0(U_s,D_\mu|_{U_s} \otimes \Omega_{U_s/(T_f)_s}^c).\]
    Since the map $\Omega_{U/T_f}^c\to D_\mu$ is a subvector bundle, so is its restriction to the fiber over $s$. In particular the above map is injective, hence $v_s|_{U_s}=0$.
    
    Since $U\subseteq X\widehat{_\mu}$ is fiberwise dense over $T_f$ and $s$ is a closed point of $T$, it follows that $U_s\subseteq  (X\widehat{_\mu})_s\simeq (X_s)\widehat{_{\mu_s}}$ is a dense open subset. Since $X_s$ is a smooth curve  over $\operatorname{Spec}k[\varepsilon]$, it follows that $(X_s)\widehat{_{\mu_s}}$ is isomorphic to $\operatorname{Spec } k[\varepsilon][[z]]$. Every zero-divisor of $k[\varepsilon][[z]]$ is nilpotent, hence, since $U_s$ is non-empty, the restriction map
    \[H^0((X_s)\widehat{_{\mu_s}},\Omega_{X_s/(T_f)_s}^{\otimes 2}|_{(X_s)\widehat{_{\mu_s}}})\to H^0(U_s,\Omega_{U_s/(T_f)_s}^{c,\otimes 2})\]
    is injective, so $v_s|_{(X_s)\widehat{_{\mu_s}}}=0$. Since $X_s$ is smooth over $\operatorname{Spec}k[\varepsilon]$, it follows that for any connected affine open subscheme $\operatorname{Spec} R'\subseteq X_s$ every zero divisor of $R'$ is nilpotent. In particular the restriction of differential forms to a connected affine open neighborhood $\operatorname{Spec } R'$ of $\mu_s$ is injective. Let $I'\subset R'$ be the ideal defining $\mu_s$ and $\mathfrak{m}'$ be a maximal ideal containing $I'$. Since every zero divisor of $R'$ is nilpotent, the map to the localization at $\mathfrak{m}'$ is injective. Hence it follows by applying Krull's intersection theorem that the map
    \[H^0(X_s,\Omega_{X_s/(T_f)_s}^{\otimes 2})\to H^0((X_s)\widehat{_{\mu_s}},\Omega_{X_s/(T_f)_s}^{\otimes 2}|_{(X_s)\widehat{_{\mu_s}}})\]
    is also injective, so $v_s=0$.
\end{proof}

\subsection{Choosing $\mu$} As the first step in proving Proposition 4.3 we construct the section $\mu$.  Since $v_{+,\eta}$ is a nilpotent absolute Higgs field on $P_{+,\eta}$ over $X_{+,\eta}$, which is a smooth curve over a field, there exist a $B$-structure $P_{B,+,\eta}$ on $P_{+,\eta}$ such that $v_{+,\eta}$ is in the kernel of the map
\[H^0(X_{+,\eta},\mathrm{At}(P_{+,\eta})^\vee \otimes \Omega_{X_{+,\eta}})\to H^0(X_{+,\eta},\mathrm{At}(P_{B,+,\eta})^\vee \otimes \Omega_{X_{+,\eta}}).\]
Similarly let $P_{B,-,\eta}$ be a $B$-structure on $P_{-,\eta}$ killing $v_{-,\eta}$. We now use the fact that $T$ is the spectrum of a discrete valuation ring to extend these $B$-structures to rational $B$-structures $P_{B+}$ and $P_{B-}$ on $P_+$ and $P_-$ respectively. 

\begin{mlem}
    There exist a fibrewise dense open subset $O_+$ of $X_+$ and a $B$-structure $P_{B,+}$ on $P_+$ over $O_+$ extending $P_{B,+,\eta}$. Furthermore, for any such $B$-structure, $v_+$ is in the kernel of the map
    \[H^0(X_+,\mathrm{At}(P_+)^\vee\otimes \Omega_{X_+/T})\to H^0(O_+,\mathrm{At}(P_{B,+})^\vee\otimes \Omega_{X_+/T}).\]
    Same for $-$ instead of $+$.
\end{mlem}
\begin{proof}
    This follows by properness of the Drinfeld compactification of $\mathrm{Bun}_B\to \mathrm{Bun}_G$ and the valuative criterion for properness since $R$ is a discrete valuation ring. A point on Drinfeld's compactification defines a $B$-structure on a fiberwise dense open subset. In the notation of Braverman-Gaitsgory \cite{braverman2000geometriceisensteinseries}, this is the open subscheme where $\kappa:\mathcal{F}_G\to \overline{G/U}\times^T\mathcal{F}_T$ factors through $G/U\times^T\mathcal{F}_T$, which is fiberwise dense by assumption. To see that $v_+$ is in the kernel of the map 
    \[H^0(X_+,\mathrm{At}(P_+)^\vee\otimes \Omega_{X_+/T})\to H^0(O_+,\mathrm{At}(P_{B,+})^\vee\otimes \Omega_{X_+/T}),\]
    notice that, since $P_{B,+}$ extends $P_{B,+,\eta}$ over $O_{+,\eta}$, $v_+$ is in the kernel after restricting further to the open set $O_{+,\eta}$. Since $O_{+,\eta}$ is dense in $O_{+}$ and $\mathrm{At}(P_{B,+})^\vee$ is a vector bundle over $O_+$, it follows that the image of $v_+$ is a generically vanishing section of a vector bundle over the reduced scheme $O_+$, so the image is $0$. 
\end{proof}

Choose $O_+$ and $O_-$ as in the lemma and a fiberwise dense open subset $O\subseteq X$ such that on $X_+$ it is contained in $O_+$ and on $X_-$ it is contained in $O_-$. Since $O\to X\to T_f$ is smooth and surjective, there exist étale locally on $T_f$ a section. Since $f\in \mathfrak{m}$, $R[x]/x(x-f)$ is a local ring and since it is finite over $R$, it follows that $R[x]/x(x-f)$ is also a strictly Henselian local ring. This implies that any étale cover of $T_f$ admits a section. Thus the smooth and surjective map $O\to T_f$ admits a section defined globally over $T_f$.  Let $\mu:T_f\to O\subseteq X$ be a choice of section. This will be the desired section $\mu$ in Proposition 4.3.

\subsection{Reduction to formal collisions}
The rest of this section will be reducing Proposition 4.3 to a formal analogue, which is independent of the curve $X$. Let $i_\mu:X\widehat{_\mu}\to X$ denote the completion of $X$ along the graph of $\mu$. Since the restriction of $\mu$ to $T_{f,+}$ factors through $O_+$, $P_{B,+}$ defines a $B$-structure $P\widehat{_{B}}_{,+}$ on $i_{\mu,+}^*P$, where $i_{\mu,+}:{X\widehat{_\mu}}_+\to X$. Similarly, since the restriction of $\mu$ to $T_{f,-}$ factors through $O_-$, $P_{B,-}$ defines a $B$-structure $P\widehat{_{B}}_{,-}$ on $i_{\mu,-}^*P$. Locally on $T_f$ we can choose a uniformization $X\widehat{_\mu}\simeq \mathrm{Spec } R[x]/x(x-f)[[z]]=:T_f[[z]]$, but since $T_f$ is the spectrum of a local ring, we can choose such a uniformization globally over $T_f$. Fix a choice of uniformization. Given this choice, we obtain a map $q:X\widehat{_\mu}\simeq T_f[[z]]\to T[[z]]$. Since $T_f[[z]]$ is the spectrum of a strictly Henselian local ring, any $G$-bundle on $T_f[[z]]$ is trivializable, so in particular we can choose a $G$-bundle $\overline{P}$ on $T[[z]]$ and an isomorphism $q^*\overline{P}\simeq i_{\mu}^*P$.

We will now use  the chosen uniformization, $G$-bundle $\overline{P}$ and isomorphism $q^*\overline{P}\simeq i_{\mu}^*P$ to extend the $B$-structure $P\widehat{_{B}}_+$ to a $B$-structure $P\widehat{_{B}}^+$ on $i_{\mu}^*P$ over all of $X\widehat{_\mu}$ and similarly for $P\widehat{_{B}}_-$. It will be important to us that we fixed one uniformization and used it to extend both $B$-structures, rather than just choosing two arbitrary extensions.

\begin{mcon} Having fixed a uniformization, we obtain the following commutative diagram.
\[
\begin{tikzcd}
X\widehat{_\mu}_+ \arrow[r] \arrow[d, "\cong"] & X\widehat{_\mu} \arrow[d, "\cong"] & X\widehat{_\mu}_- \arrow[l] \arrow[d, "\cong"] \\
{T_{f+}[[z]]} \arrow[rd, "\cong"] \arrow[r, "i_+"]          & {T_f[[z]]} \arrow[d, "q"]          & {T_{f-}[[z]]} \arrow[ld, "\cong"'] \arrow[l, "i_-"']        \\
                                                            & {T[[z]]}                           &                                                            
\end{tikzcd}\]
We will use the same notation for $i_\mu^* P$, $i_{\mu+}^* P$ and $i_{\mu -}^* P$ and their image under the three vertical isomorphisms, so $i_+^*i_\mu^*P\simeq i_{\mu+}^*P$. From the isomorphism $q^*\overline{P}\simeq i_\mu^* P$ we obtain an isomorphism $i_+^* q^*\overline{P}\simeq i_{\mu,+}^* P$ over $T_{f+}[[z]]$ and similarly $i_-^*q^*\overline{P}\simeq i_{\mu,-}^* P$ over $T_{f-}[[z]]$. Pulling back these isomorphisms along $(q\circ i_+)^{-1}$ and $(q\circ i_-)^{-1}$, we obtain isomorphisms 
\[ ((q\circ i_+)^{-1})^* i_{\mu +}^*P\simeq \overline{P}\simeq ((q\circ i_-)^{-1})^* i_{\mu -}^*P.\]
Thus the $B$-structure $P\widehat{_{B}}_{,+}$ on $i_{\mu,+}^* P$ provides a $B$-structure on $\overline{P}$. Similarly $P\widehat{_{B}}_{,-}$ provides a second $B$-structure on $\overline{P}$. Pulling these $B$-structures back along $q$ and using the isomorphism $q^*\overline{P}\simeq i_\mu^*P$, we obtain the two desired $B$-structures $P\widehat{_{B}}^+$ and $P\widehat{_{B}}^-$ on $i_\mu^*P$. From the construction it follows that pulling back $P\widehat{_{B}}^+$  along $i_{+}$ gives $P\widehat{_{B}}_{,+}$ and similarly pulling back $P\widehat{_{B}}^-$ along $i_{-}$ gives $P\widehat{_{B}}_{,-}$.
\end{mcon}

The compatibility between the two extensions that will be important in the next section is that the two $B$-structures $P\widehat{_{B}}^+$ and $P\widehat{_{B}}^-$ on $i_\mu^*P$ are pulled back along $q$ from two $B$-structures on a single $G$-bundle over $T[[z]]$.

\subsection{Statement for formal collisions}

Note that, once $\mu$ is chosen, the statement of Proposition 4.3 only takes place in the formal completion. So, having fixed a uniformization and using Proposition 2.5, which states that the pullback of the dual Atiyah bundle is the completed dual Atiyah bundle, we have eliminated $X$ from the problem. Let us summarize the data we have and take the opportunity to simplify the notation. What is left to prove of Proposition 4.3 is reformulated in this formal situation in Proposition 4.6.

Let $T=\operatorname{Spec} R$ with $R$ a strictly Henselian discrete valuation ring and $f\in \mathfrak{m}\setminus\{0\}$. We have the diagram 
\[
\begin{tikzcd}
{T_{f+}[[z]]} \arrow[rd, "\cong"] \arrow[r, "i_+"] & {T_f[[z]]} \arrow[d, "q"] & {T_{f-}[[z]]} \arrow[ld, "\cong"'] \arrow[l, "i_-"'] \\
                                                   & {T[[z]]}                  &                                                     
\end{tikzcd}\]
and $G$-bundle $\overline{P}$ on $T[[z]]$ with two $B$-structures $\overline{P}_B^+$ and $\overline{P}_B^-$. Let $P=q^*\overline{P}$ and $P_B^+$ and $P_B^-$ denote the pullback of the two $B$-structures. Then we have the absolute Higgs field $v\in H^0(T_f[[z]],\mathrm{At}(P)^{\vee,c} \otimes \Omega_{T_f[[z]]/T_f})$ such that $i_+^*v$ is in the kernel of the map
\[H^0(T_{f+}[[z]],\mathrm{At}(P)^{\vee,c} \otimes \Omega_{T_f[[z]]/T_f}^c)\to H^0(T_{f+}[[z]],\mathrm{At}(P_B^+)^{\vee,c} \otimes \Omega_{T_f[[z]]/T_f}^c)\]
and similarly $i_-^*v$ is in the kernel of the map
\[H^0(T_{f-}[[z]],\mathrm{At}(P)^{\vee,c} \otimes \Omega_{T_f[[z]]/T_f}^c)\to H^0(T_{f-}[[z]],\mathrm{At}(P_B^-)^{\vee,c} \otimes \Omega_{T_f[[z]]/T_f}^c).\]
To prove Proposition 4.3 it suffices to prove the following statement, which no longer refers to the curve $X$.

\begin{mprop}
    In the situation above, there exist an open subset $U\subseteq T_f[[z]]$, which is fiberwise dense over $T_f$, and a quotient vector bundle $D$ of $\mathrm{At}(P)^{\vee,c}$ such that the composition $\Omega_{U/T_f}^c\to \mathrm{At}(P)^{\vee,c} \to D$ defines a subvector bundle of $D$ and $v$ is in the kernel of the map
    \[H^0(T_f[[z]],\mathrm{At}(P)^{\vee,c} \otimes \Omega_{T_f[[z]]/T_f})\to H^0(U,D \otimes \Omega_{U/T_f}^c)\]
\end{mprop}

\section{Formal collisions for groups of semi-simple rank 1}

Throughout this section, let $G$ be a reductive group over $k$ of semi-simple rank 1. Let $T\subset B\subseteq G$ be a choice of Borel subgroup and a maximal torus, $\mathbb{X}_\bullet(B)$ denote the characters of $B$ and $W$ be the Weyl group of $G$. Let $\alpha\in \mathbb{X}_\bullet(B)$ be the character given by the unique positive simple root and $w\in W$ be the unique non-trivial Weyl element. So $w\cdot \alpha =-\alpha$. For each $B$-bundle $P_B$ and $\chi\in \mathbb{X}_\bullet(B)$, let $\chi (P_B)$ denote the induced $\mathbb{G}_m$-bundle.

In the section we prove Proposition 4.6 and use the notation introduced in Section 4.5. Note that if the two $B$-structures $\overline{P}_B^+$ and $\overline{P}_B^-$ are equal over all of $T[[z]]$, then we can take $D=\mathrm{At}(P_B^+)^\vee$. Thus we can assume that there is a dense, but not necessarily fiberwise dense over $T$, open subset $U'$ of $T[[z]]$ where the two $B$-structures are fiberwise different. Since $G$ has semi-simple rank 1, fiberwise different implies fiberwise opposed, that is the $B$-structures refine to a $T$-structure over $U'$.

\begin{mdef}
    Let $I$ denote the coherent sheaf on $T_f[[z]]$ given by the following pushout diagram
     \[
\begin{tikzcd}
\mathrm{At}(P)^{\vee,c} \arrow[r] \arrow[d] & \mathrm{At}(P_B^+)^{\vee,c} \arrow[d] \\
\mathrm{At}(P_B^-)^{\vee,c} \arrow[r]       & I                                
\end{tikzcd}\]
For each character $\chi\in \mathbb{X}_\bullet(B)$, let $\rho_{\chi}^+:\mathrm{At}(\chi (P_B^+))^{\vee,c}\to I$ denote the composition $\mathrm{At}(\chi (P_B^+))^{\vee,c}\to \mathrm{At}(P_B^+)^{\vee,c}\to I$. Define $\rho_\chi^-$ similarly. 
\end{mdef}

In general $I$ is not a vector bundle and we will construct $D$ as a further quotient of $I$. To do this we first consider the relative setting, that is when the dual Atiyah bundles are replaced by the coadjoint bundles.

\begin{mdef}
    Let $I_\mathrm{rel}$ denote the coherent sheaf on $T_f[[z]]$ given by the following pushout diagram
    \[
\begin{tikzcd}
\mathfrak{g}^\vee_{P} \arrow[r] \arrow[d] & \mathfrak{b}^\vee_{ P_B^+} \arrow[d] \\
\mathfrak{b}^\vee_{ P_B^-} \arrow[r]  & I_\mathrm{rel}                          
\end{tikzcd}\]
For each character $\chi\in \mathbb{X}_\bullet(B)$, let $\rho_{\chi,\mathrm{rel}}^+:\mathcal{O}_{T_f[[z]]}\to I_\mathrm{rel}$ denote the composition $\mathcal{O}_{T_f[[z]]}\cong (\mathfrak{gl}_1)^\vee_{\chi P_B^+}\to \mathfrak{b}^\vee_{ P_B^+}\to I_\mathrm{rel}$. Define $\rho_{\chi,\mathrm{rel}}^-$ similarly.
\end{mdef}
Note that the coadjoint bundle of a $\mathbb{G}_m$-bundle is canonically trivial, this is not the case for the dual Atiyah bundle of a $\mathbb{G}_m$-bundle. This is allows us to make the following definition in the setting of relative Higgs fields, which does not have an immediate analogue for absolute Higgs fields.

\begin{mdef}
    Let $D_\mathrm{rel}$ denote the quotient coherent sheaf of $I_\mathrm{rel}$ given by the coequalizer $I_\mathrm{rel}\to \mathrm{coeq}(\rho_{\alpha,\mathrm{rel}}^+,\rho_{w\cdot \alpha,\mathrm{rel}}^-)$, where $\alpha$ is character of $B$ given by the unique positive simple root.
\end{mdef}

We now verify that this quotient is indeed a vector bundle.

\begin{mprop}
    The coherent sheaf $D_\mathrm{rel}$ is a vector bundle on $T_f[[z]]$.
\end{mprop}
\begin{proof}
    Since $T_f[[z]]$ is Noetherian and reduced it suffices to prove that all fibers of $D_\mathrm{rel}$ have the same dimension. Let $x$ be a point of $T_f[[z]]$.

    First suppose the two $B$-structures $P_B^+$ and $P_B^-$ on $P$ are equal in the fiber at $x$. Since the fiber of a cokernel of coherent sheaves is the cokernel of the fiber, the fiber of $D_\mathrm{rel}$ above $x$ identifies with the cokernel of the map
    \[ (\mathfrak{gl}_1)^\vee \to \mathfrak{b}^\vee \]
    induced by $\alpha-w\cdot \alpha=2\alpha$. Thus the fiber has dimension $\operatorname{dim} B-1=\operatorname{dim} T$.

    On the other hand, if the two $B$-structures $P_B^+$ and $P_B^-$ on $P$ are different in the fiber at $x$, then they refine to a $T$-structure $P_T$ with an isomorphism of the induced $B$-bundle with $P_B^+$ and the induced $wBw^{-1}$-bundle with $P_B^{-}$, where $B$ is identified with $wBw^{-1}$ by conjugation with $w$. In this case we see that $I_\mathrm{rel}=\mathfrak{t}_{P_T}^\vee$, where $\mathfrak{t}$ is the Lie algebra of $T$. Since the following diagram commutes
    \[
\begin{tikzcd}
                                  & \mathfrak{b} \arrow[rd, "\alpha"]                  &                 \\
\mathfrak{t} \arrow[r] \arrow[ru] & w\mathfrak{b}w^{-1} \arrow[d, "\cong"] \arrow[r] & \mathfrak{gl}_1 \\
                                  & \mathfrak{b} \arrow[ru, "w\cdot\alpha"']           &                
\end{tikzcd}\]
    it follows that $\mathfrak{t}_{P_T}^\vee$ already coequalizes all the map $\rho_{\alpha,\mathrm{rel}}^+$ and $\rho_{w\cdot \alpha,\mathrm{rel}}^-$ at $x$.
    Hence also in this case the fiber of $D_\mathrm{rel}$ at $x$ has dimension $\text{dim } \mathfrak{t}=\text{dim }T$.
\end{proof}

As mentioned before, the problem with making this definition of $D_\mathrm{rel}$ work in the absolute setting is that the dual Atiyah bundles of $\mathbb{G}_m$-bundles are not canonically isomorphic. This is where the fiberwise dense open subset $U\subseteq T_f[[z]]$ in Proposition 4.6 comes into the picture.

On the dense open subset of $T[[z]]$ where the $B$-structures $\overline{P}_B^+$ and $\overline{P}_B^-$ are opposed, we have an isomorphism between $\alpha (\overline{P}_B^+)$ and $(w\cdot \alpha) (\overline{P}_B^-)$.  However the problem is that this set is not necessarily fiberwise dense over $T$. The trick is that one can nonetheless uniquely extend the generic isomorphism induced on the dual Atiyah bundles to an isomorphism over an open subset which is fiberwise dense over $T$. To do this we need the following lemma.

\begin{mlem}
    Let $T[[z]]=\mathrm{Spec } R[[z]]$ with $R$ a discrete valuation ring, let $L$ and $Q$ be $\mathbb{G}_m$-bundles over $T[[z]]$ and $h:L|_{U'}\to Q|_{U'}$ be an isomorphism defined over a non-empty open subscheme $U'\subset T[[z]]$. There exist an open subset $U\subseteq T[[z]]$ containing $U'$, which is fiberwise dense over $T$, and a unique isomorphism $\mathrm{At}|_U:\mathrm{At}(Q|_U)^{\vee,c}\overset{\sim}{\to} \mathrm{At}(L|_U)^{\vee,c} $ under $\Omega_{T[[z]]/T}^c$, which extends the induced $\mathrm{At}(h):\mathrm{At}(Q|_{U'})^{\vee,c}\overset{\sim}{\to} \mathrm{At}(L|_{U'})^{\vee,c}$ on the open subscheme $U'$.
\end{mlem}

\begin{proof}
    We start by assuming that $U'$ is a distinguished open subset, let $U'=D(g)$ for some $g\in R[[z]]$. Since every $\mathbb{G}_m$-bundle on $T[[z]]$ is trivializable, we may assume that $L$ and $Q$ are trivial. Thus the map $h$ corresponds to a unit $a/g^n\in R[[z]]_g$. By Lemma 2.6, the map $\mathrm{At}(h)$ is given by $\mathrm{At}(Q|_{U'})^{\vee,c} =\mathcal{O}_{U'}\oplus \Omega_{U'/T}^c\overset{\sim}{\to} \mathcal{O}_{U'}\oplus \Omega_{U'/T}^c=\mathrm{At}(L|_{U'})^{\vee,c}$, where the middle map is given by the following matrix
    \[\begin{pmatrix}
        1& \mathrm{dlog} (a/g^n)\\
        0&1
    \end{pmatrix}.\]
    Thus to extend $\mathrm{At}(h)$ to a fiberwise dense open subset is equivalent to extending $\mathrm{dlog} (a/g^n)$ to a fiberwise dense open subscheme.  Let $t\in R$ be a uniformizer. Since $U'$ is non-empty, the function $g$ is not identically $0$, thus we can write $g^n=t^k g'$, for some $g'\in R[[z]]$, where $g'$ is a power series with at least one coefficient not divisible by $t$. Since $h$ is a unit in $R[[z]]_g$, so $a$ also not identically zero, we can write $a=t^la'$ where $a'\in R[[z]]$ is a power series with at least one coefficient not divisible by $t$. 
    
    Then we have the following equalities of differential forms over $U'=D(g)$
    \[\mathrm{dlog} (a/g^n)=\frac{d(a/g^n)}{a/g^n}=\frac{d(t^l a'/(t^kg'))}{t^l a'/(t^kg')}=\frac{t^{l-k}d( a'/g')}{ t^{l-k}a'/g'}=\frac{d( a'/g')}{ a'/g'}=\mathrm{dlog }(a'/g').\]
    Let $U=D(g'a')$. Then $U$ contains $U'$ and by the above equalities $\mathrm{dlog }(a'/g')$ is a differential from on $U$ which restricts to $\mathrm{dlog} (a/g^n)$ over $U'$. Since at least one coefficient of both $a'$ and $g'$ is not divisible by $t$, at least one coefficient of $a'g'$ is not divisible by $t$. This implies that $U$ is dense in the special fiber over $T$, hence since $U$ is open and $T$ is the spectrum of a discrete valuation ring, it follows that $U$ is fiberwise dense over $T$. 
    
    The uniqueness follows since $T[[z]]$ is an integral domain, hence, given $U$, a differential form on $U$ is uniquely determined by its restriction to a non-empty open subset.
\end{proof}

Applying Lemma 5.5 over $T[[z]]$ and pulling back along the map $q:T_f[[z]]\to T[[z]]$, we obtain an open subset $U\subseteq T_f[[z]]$ which is fiberwise dense over $T_f$ and isomorphisms $c_{\alpha}:\mathrm{At}(\alpha P_B^+)^{\vee,c}\cong \mathrm{At}((w\cdot \alpha) P_B^-)^{\vee,c}$ over $U$ extending the previously defined isomorphism over $U'$ to a fiberwise dense set. This allows us to make the following definition.

\begin{mdef}
    Let $D$ be the coherent quotient sheaf of $I|_U$ over $U$ given by the coequalizer $I|_U\to \mathrm{coeq}(\rho_{\alpha}^+,\rho_{w\cdot \alpha}^- \circ c_{\alpha})$.
\end{mdef}

Consider the map $\Omega_{U/T_f}^c\to I|_U\to D$. Since $c_{\alpha}:\mathrm{At}(\alpha P_B^+)^{\vee,c}\cong \mathrm{At}((w\cdot \alpha) P_B^-)^{\vee,c}$ is a map under $\Omega_{U/T_f}^c$ and identifies with the canonical isomorphism $(\mathfrak{gl}_1)^\vee_{\alpha P_B^+}\cong \mathcal{O}_U\cong (\mathfrak{gl}_1)^\vee_{w\cdot \alpha P_B^-}$ on the quotient, it follows that the cokernel of the map $\Omega_{U/T_f}^c\to D$ is $D_\mathrm{rel}|_U$.

\begin{mprop}
     The sequence $0\to \Omega_{U/T_f}^c\to D\to D_\mathrm{rel}|_U\to 0$ is exact and in particular $D$ is a vector bundle over $U$.
\end{mprop}
\begin{proof}
    Let $U'$ denote the dense open subset of $U$ where the two $B$-structures are opposed. Over $U'$ we obtain a $T$-structure $P_T$ on $P$ such that the bundle obtained by inducing up against $T\subset B$ identifies with $P_B^+$ and inducing up against $T\subset wBw^{-1}$ gives $P_B^-$. In particular we obtain a commutative diagram
    \[ 
\begin{tikzcd}
                                   & \mathrm{At}(P)^{\vee,c }\arrow[ld] \arrow[rd] &                                    \\
\mathrm{At}(P_B^+)^{\vee,c} \arrow[rd] &                                           & \mathrm{At}(P_B^-)^{\vee,c} \arrow[ld] \\
                                   & \mathrm{At}(P_T)^{\vee,c}                     &                                   
\end{tikzcd}\]
    which gives a map $I\to \mathrm{At}(P_T)^{\vee,c}$ over $U'$. Since the  map  $I\to \mathrm{At}(P_T)^{\vee,c}$ coequalizes $\rho_\alpha^+$ and $\rho_{w\cdot \alpha}^- \circ c_{\alpha}$, we obtain a map $D\to \mathrm{At}(P_T)^{\vee,c}$ over $U'$. By construction of the dual Atiyah bundle, the map $\Omega^c\to \mathrm{At}(P_T)^\vee$ is injective on fibers. So the map $\Omega^c \to D$ is also injective on fibers over $U'$ and the sequence $0\to \Omega_{U/T_f}^c\to D\to D_\mathrm{rel}|_U\to 0$ is exact on $U'$. In particular it follows that $D$ has rank $\text{dim } \mathfrak{t}+1$ generically. Since $D$ is coherent and $T_f[[z]]$ is Noetherian and reduced, the rank of $D$ can only increase under specialization. So for every point of $U$, the rank of $D$ is at least $\text{dim }\mathfrak{t}+1$.

    Let $x$ be a point of $U$. Since the fiber of the cokernel is the cokernel of the fibers we obtain an exact sequence
    \[ \Omega_{U/T_f}^c|_{k(x)}\to D|_{k(x)}\to D_\mathrm{rel}|_{k(x)}\to 0\]
    Since $D$ has fibers of dimension at least $\text{dim }\mathfrak{t}+1$, $\Omega_{U/T_f}^c|_{k(x)}$ has dimension $1$ and $D_\mathrm{rel}|_{k(x)}$ has dimension $\text{dim }\mathfrak{t}$, it follows that $D|_{k(x)}$ has dimension exactly $\text{dim }\mathfrak{t}+1$ and that the above sequence is also exact to the left. Since $D$ is a coherent sheaf and $T_f[[z]]$ is reduced and Noetherian, it follows that $D$ vector bundle. Since the sequence of vector bundles is exact on fibers, it is a short exact sequence of vector bundles.
\end{proof}
This gives the desired fiberwise dense open subset $U\subseteq T_f[[z]]$ and quotient bundle $\mathrm{At}(P)^{\vee, c}\to D$ such that $\Omega^c\to D$ is a subvector bundle. The last thing to verify to prove Proposition 4.6 is that our given absolute Higgs field $v$ is in the kernel of the map $H^0(T_f[[z]],\mathrm{At}(P)^{\vee,c}\otimes \Omega_{T_f[[z]]/T_f}^c)\to H^0(U,D\otimes \Omega_{T_f[[z]]/T_f}^c)$. 

\begin{mprop}
The absolute Higgs field $v$ is in the kernel of the map
\[H^0(T_f[[z]],\mathrm{At}(P)^{\vee,c} \otimes \Omega_{T_f[[z]]/T_f}^c)\to H^0(U,D \otimes \Omega_{T_f[[z]]/T_f}^c).\]
\end{mprop}
\begin{proof}
    By construction of $D$, the map factors as
    \[H^0(T_f[[z]],\mathrm{At}(P)^{\vee,c} \otimes \Omega_{T_f[[z]]/T_f}^c)\to H^0(T_f[[z]],\mathrm{At}(P_B^+)^{\vee,c} \otimes \Omega_{T_f[[z]]/T_f}^c) \to H^0(U,D \otimes \Omega_{T_f[[z]]/T_f}^c)\]
    By assumption $i_+^*v$ is in the kernel of the first map after restricting to $T_{f+}[[z]]$. Hence the image of $v$ in $H^0(U,D \otimes \Omega_{T_f[[z]]/T_f}^c)$ is fiberwise $0$ on $T_{f+}[[z]]$. Similarly by construction of $D$, the map factors as
    \[H^0(T_f[[z]],\mathrm{At}(P)^{\vee,c} \otimes \Omega_{T_f[[z]]/T_f}^c)\to H^0(T_f[[z]],\mathrm{At}(P_B^-)^{\vee,c} \otimes \Omega_{T_f[[z]]/T_f}^c) \to H^0(U,D \otimes \Omega_{T_f[[z]]/T_f}^c).\]
    Again by assumption $i_-^*v$ is in the kernel of the first map after restricting to $T_{f-}[[z]]$. Hence the image of $v$ in $H^0(U,D \otimes \Omega_{T_f[[z]]/T_f}^c)$ is fiberwise $0$ on $T_{f-}[[z]]$. Thus the image of $v$ is a global section of $D\otimes \Omega_{U/T_f}^c$ which vanishes in each fiber over $T_f[[z]]$. Since $T_f[[z]]$ is reduced, this implies that the image of $v$ is $0$.
\end{proof}
This finishes the proof of Proposition 4.6 and hence also Proposition 4.3. Thus we conclude Proposition 1.1.

\begin{proof}[Proof of Corollary 1.2]
    Since non-characteristicness can be checked smooth locally on $\mathrm{Bun}_G$, it suffices to consider $f^*\mathcal{H}om(\mathcal{F},\mathcal{G})$ and $f^*(\mathcal{F}\otimes \mathcal{G})$ for any smooth morphism $f:U\to \mathrm{Bun}_G$ with $U$ a scheme finite type over $k$. Since $f$ is smooth $f^!(-)=f^*(-)\otimes f^!(1)$ and by the projection formula
    \begin{align*}
        f^*\mathcal{H}om(\mathcal{F},\mathcal{G})&=f^!\mathcal{H}om(\mathcal{F},\mathcal{G}) \otimes f^!(1)^{\otimes(-1)}\\
        &=
        \mathcal{H}om(f^*\mathcal{F}, f^*\mathcal{G}\otimes f^!(1))\otimes f^!(1)^{\otimes(-1)}=\mathcal{H}om(f^*\mathcal{F},f^*\mathcal{G}).
    \end{align*}
    By symmetric monoidality $f^*(\mathcal{F}\otimes \mathcal{G})=f^*\mathcal{F}\otimes f^*\mathcal{G}$. Now $U$ is a smooth variety over $k$, hence it follows by Kashiwara and Schapira's estimate for the singular support of $\mathcal{H}om(-,-)$ and $-\otimes-$ and Proposition 1.1 that these sheaves on $U$ are non-characteristic over $\mathcal{M}_g$. 
\end{proof}

\section{Counterexamples for Higher Rank Groups}

In this section we will construct a family of curves $X\to T=\operatorname{Spec} R$ of genus $4$, where $R$ is a discrete valuation ring, and an absolute Higgs field $v$ on a $\mathrm{GL}_3$-bundle $P$ over $X$, which on the generic fiber is the sum of two nilpotent absolute Higgs fields and on the special fiber is given by a non-zero element in the subspace
\[H^0(X_s,\Omega_{X_s}^{\otimes 2})\to H^0(X_s,\mathrm{At}(P|_{X_s})^\vee\otimes \Omega_{X_s}).\]
Thus providing a counterexample to the analogue of Proposition 1.1 in this case. The example immediately generalizes to $\mathrm{GL}_n$ for $n\ge 3$ and a curve of genus $g\ge 4$.

Let $t\in R$ be a uniformizer, $k$ be the residue field, and $K$ be the field of fractions.

We start by providing a formal counterexample, which we will then algebraize later.

\subsection{A formal counterexample}

We want to find an example of an absolute formal Higgs field $v$ over the trivial $\mathrm{GL}_3$-bundle $P_G$ on $\operatorname{Spec } R[[z]]$, which when restricted to the generic fiber $R[[z]]\otimes_R K$ is a sum of two nilpotent absolute Higgs field and when restricted to the special fiber $k[[z]]$ is a non-zero element in the image of the map
\[H^0(\operatorname{Spec } k[[z]],\Omega^{c}_{k[[z]]/k})\to H^0(\operatorname{Spec } k[[z]],\mathrm{At}(P_G)^{\vee,c}).\]
In the formal setting we will refer to elements of $H^0(\operatorname{Spec }R[[z]],\mathrm{At}(P_G)^{\vee,c})$ as absolute Higgs fields, since we have a canonical choice of trivialization of the sheaf of differential forms this difference does not matter.

Since $P_G$ is the trivial $\mathrm{GL}_3$-bundle, the dual Atiyah bundle splits as $\mathrm{At}(P_G)^{\vee,c}=\mathfrak{gl}_3^\vee\oplus \Omega^c$, where $\mathfrak{gl}_3^\vee$ is considered as a trivial vector bundle over $R[[z]]$. Let $v\in H^0(\operatorname{Spec } R[[z]],\mathrm{At}(P_G)^{\vee,c})$ be the absolute Higgs field given by the pair $(N,\omega)$ where $N\in H^0(\operatorname{Spec } R[[z]],\mathfrak{gl}_3^\vee)$ is given by the matrix
\[N=\begin{pmatrix}
    0&-tz(z-1)&t^2z(z-1)\\
    0&0&tz(z-1)\\
    0&0&0
\end{pmatrix}\]
and $\omega\in H^0(T[[z]],\Omega^c)$ is the differential from $-dz$. Clearly $N\equiv 0 \text{ mod }t $ and $-dz\not\equiv 0 \text{ mod }t$, so the only thing to prove is that $(N,\omega)$ is generically a sum of two nilpotent absolute Higgs fields. To do this consider the pairs $(N_0,\omega_0)$ and $(N_1,\omega_1)$ defined over $\operatorname{Spec }R[[z]]\otimes_R K$, where $N_0$ is given by the matrix
\[N_0=\begin{pmatrix}
    0&0&0\\
    t^{-1}(z-1)&0&0\\
    -t^{-2}&t^{-1}z&0
\end{pmatrix}\]
and $\omega_0=0$, and $N_1$ is given by the matrix
\[N_1=\begin{pmatrix}
    0&-tz(z-1)&t^2z(z-1)\\
    -t^{-1}(z-1)&0&tz(z-1)\\
    t^{-2}&-t^{-1}z&0
\end{pmatrix}\]
and $\omega_1=-dz$. It is clear that $(N,\omega)=(N_0,\omega_0)+(N_1,\omega_1)$ over $\operatorname{Spec } R[[z]]\otimes_R K$. It remains to show that $(N_0,\omega_0)$ and $(N_1,\omega_1)$ are nilpotent as absolute Higgs fields. To see that $(N_0,\omega_0)$ is nilpotent note that it is killed by the standard $B$-structure on the trivial $\mathrm{GL}_3$-bundle, where $B$ is the Borel subgroup  of upper triangular matrices, so it is nilpotent. To see that $(N_1,\omega_1)$ is nilpotent, we will show that it is killed by the $B$-structure given by translating the standard $B$-structure by the matrix
\[A=\begin{pmatrix}
    1&0&0\\
    t&1&0\\
    t^2z&t&1
\end{pmatrix}\in \mathrm{GL}_3(R[[z]]).\]
We have to determine the image of $(N_1,\omega_1)$ under the map
\[ H^0(\mathrm{Spec } R[[z]],\mathrm{At}(P_G)^{\vee,c})\to  H^0(\mathrm{Spec } R[[z]],\mathrm{At}(P_G)^{\vee,c})\] 
induced by $A$ on the dual Atiyah bundles. This map sends a pair $(N_1,\omega_1)$ to the pair given by  $(\mathrm{ad}^\vee(A) N_1, \omega_1+m_A(N_1))$, where $\mathrm{ad}^\vee(A):\mathfrak{gl}_3^\vee\to \mathfrak{gl}_3^\vee $ is the coadjoint action by $A$ and $m_A:\mathfrak{gl}_3^\vee \to H^0(T[[z]],\Omega^c)$ is given by pulling back the invariant differential forms along the map $A:T[[z]]\to \mathrm{GL}_3$. 

The image of $N_1$ under the coadjoint action by $A$ is given by the 
matrix
\begin{align*}
    \mathrm{ad}^\vee(A)N_1=&\begin{pmatrix}
    1&-t&-t^2z+t^2\\
    0&1&-t\\
    0&0&1
\end{pmatrix}\begin{pmatrix}
    0&-tz(z-1)&t^2z(z-1)\\
    -t^{-1}(z-1)&0&tz(z-1)\\
    t^{-2}&-t^{-1}z&0
\end{pmatrix}\begin{pmatrix}
    1&t&t^2z\\
    0&1&t\\
    0&0&1
\end{pmatrix}\\
=&\begin{pmatrix}
    0& 0&0\\
    -t^{-1}z&z&tz(z-1)\\
    t^{-2}&-t^{-1}z&0
\end{pmatrix}\begin{pmatrix}
    1&t&t^2z\\
    0&1&t\\
    0&0&1
\end{pmatrix}\\
=&\begin{pmatrix}
    0&0&0\\
    -t^{-1}z&0&0\\
    t^{-2} & -t^{-1}(z-1)&0
\end{pmatrix}
\end{align*}
This matrix is killed by the standard $B$-structure, so it only remains to show that $\omega_1+m_A(N_1)=0$ or equivalently that $m_A(N_1)=dz$. Let $M$ denote the matrix entry coordinates on $\mathrm{GL}_3$, the left invariant $1$-forms on $\mathrm{GL}_3$ is given by the matrix entries of $M^{-1}dM$, thus after pulling back along the map $A:T[[z]]\to \mathrm{GL}_3$, we are interested in the matrix entries of 
\[A^{-1}dA=\begin{pmatrix}
    1&0&0\\
    -t&1&0\\
    -t^2z+t^2&-t&1
\end{pmatrix}\begin{pmatrix}
    0&0&0\\
    0&0&0\\
    t^2dz&0&0
\end{pmatrix}=\begin{pmatrix}
    0&0&0\\
    0&0&0\\
    t^2dz&0&0
\end{pmatrix}.\]
So we see that $m_A(N_1)=t^{-2}t^2dz=dz$.

\subsection{Algebraizing the formal collision}

We wish to upgrade the example of a characteristic formal collision of nilpotent absolute Higgs fields to a collision over a smooth proper family of curves. To do this, let $X$ be a family of hyperelliptic curves of genus 4 over $\mathrm{Spec } R$. The property of such families that will be important to us is that the canonical morphism of $X$ factors as a finite morphism $f:X\to \mathbb{P}^1_R$ followed by the degree $3$ Veronese embedding $\mathbb{P}^1_R\to \mathbb{P}^3_R$. In other words, $f$ comes with an isomorphism $f^*\mathcal{O}_{\mathbb{P}^1_R}(3)\cong \Omega_{X/T}^1$.

The strategy is now the following, produce a $\mathrm{GL}_3$-bundle $V$ over $\mathbb{P}^1_R$ and an element $v\in H^0(\mathbb{P}^1_R,\mathrm{At}(V)^\vee\otimes \mathcal{O}(3))$ such that $v=v_0+v_1$ over the generic fiber $\mathbb{P}^1_K$, where $v_0,v_1\in H^0(\mathbb{P}^1_K,\mathrm{At}(V)^\vee\otimes \mathcal{O}(3))$ are nilpotent, and over the special fiber $v$ is a non-zero element in the subspace $H^0(\mathbb{P}^1_k,\Omega_{\mathbb{P}^1_k/k}\otimes \mathcal{O}(3))$. Pulling back along $f$ and using the identification $f^*\mathcal{O}(3)\cong \Omega_{X/T}$ this data gives rise to an absolute Higgs field over $X$ which is generically the sum of two nilpotent Higgs fields and over the special fiber is given by a non-zero characteristic vector. 

Over $\mathbb{P}^1_R$ the constructions will live on the rank $3$ bundle $V=\mathcal{O}(-1)\oplus \mathcal{O}^2$. The maps from $V^\vee$ to $V^\vee(3)= V^\vee\otimes \mathcal{O}(3)$ are given by matrices of the from
\[\begin{pmatrix}
    \mathcal{O}(3)& \mathcal{O}(4) &\mathcal{O}(4)\\
    \mathcal{O}(2) & \mathcal{O}(3)& \mathcal{O}(3)\\
    \mathcal{O}(2) &\mathcal{O}(3)&\mathcal{O}(3)
\end{pmatrix},\]
where we have ordered the factors such that the first factor is $\mathcal{O}(-1)^\vee=\mathcal{O}(1)$ and the last two factor are $\mathcal{O}^\vee=\mathcal{O}$. Let $z$ denote the coordinate on $\mathbb{P}^1_R$ and let $1,z\in H^0(\mathbb{P}^1_R,\mathcal{O}(1))$ be the corresponding basis. Over the generic fiber $\mathbb{P}^1_K$ consider the map $N_0:V^\vee\to V^\vee(3)$ given by the matrix
\[N_0=\begin{pmatrix}
    0&0&0\\
    t^{-1}(z-1)&0&0\\
    -t^{-2}&t^{-1}z&0
\end{pmatrix}\]
This is nilpotent and killed by the standard flag. Consider also the map $N_1:V^\vee\to V^\vee(3)$ over $\mathbb{P}^1_K$ given by
\[N_1=\begin{pmatrix}
    0&-tz(z-1)&t^2z(z-1)\\
    -t^{-1}(z-1)&0&tz(z-1)\\
    t^{-2}&-t^{-1}z&0
\end{pmatrix}\]
This matrix is also nilpotent and by the same calculation as in the formal setting it is killed by the flag given by translating the standard flag by following automorphism of $V$
\[A=\begin{pmatrix}
    1&0&0\\
    t&1&0\\
    t^2z&t&1
\end{pmatrix}\in \mathrm{Aut}(V).\]
For any flag $F_\bullet$ on $V$ the following diagram is a pullback
\[
\begin{tikzcd}
\mathrm{At}(V)^\vee \otimes \mathcal{O}(3) \arrow[d] \arrow[r] & \mathrm{At}(F_\bullet)^\vee \otimes \mathcal{O}(3) \arrow[d] \\
\mathrm{End}(V)^\vee\otimes \mathcal{O}(3) \arrow[r]           & \mathrm{End}(F_\bullet)^\vee\otimes \mathcal{O}(3)          
\end{tikzcd}\]
Since $N_0$ and $N_1$ are nilpotent, they are killed by a pair of flags and hence we may lift them uniquely to elements of $v_0,v_1\in H^0(\mathbb{P}^1_K,\mathrm{At}(V)^\vee \otimes \mathcal{O}(3))$, which are killed by the same respective flags. Now we have to show that $v_0+v_1$ extends to a section over all of $\mathbb{P}^1_R$, which over the special fiber $\mathbb{P}^1_k$ is a non-zero vector in the subspace $H^0(\mathbb{P}^1_k,\Omega_{\mathbb{P}_k^1/k}\otimes \mathcal{O}(3))$.

Consider the cover of $\mathbb{P}^1_R=\operatorname{Spec } R[z]\cup \operatorname{Spec } R[z^{-1}]$ given by the coordinate $z$. It suffices to prove the statement over $\operatorname{Spec} R[z]$ and $\mathrm{Spec } R[z^{-1}]$. Over $\operatorname{Spec } R[z]$ we have a trivialization of $V$ and $N_0$ is killed by the standard flag, so $(N_0,0)\in H^0(\operatorname{Spec} K[z],\mathrm{End}(V)^\vee \oplus \mathcal{O}dz)$ is a nilpotent element lifting $N_0$ so $v_0=(N_0,0)$. On the other hand $N_1$ is killed by the the flag given by translating the standard flag by $A$. Hence it follows by the same calculations as in the formal setting that $v_1=(N_1,-dz)$ is the nilpotent lift of $N_1$. From this it follows that $v=v_0+v_1=(N_0+N_1,-dz)$ over $\operatorname{Spec } K[z]$. This extends to all of $\operatorname{Spec } R[z]$ and on the special fiber it is given by a non-zero differential form. Exactly the same calculation in the coordinate $z^{-1}$ shows that this is true over all of $\mathbb{P}^1_R$.

Let $V_X=f^*V$, pulling back $v$ along the map $f:X\to \mathbb{P}^1_R$ and using the isomorphism $f^*\mathcal{O}(3)\simeq \Omega_{X/T}$, we obtain an element of $H^0(X,\mathrm{At}(V_X)^\vee\otimes f^*\mathcal{O}(3))\cong H^0(X,\mathrm{At}(V_X)^\vee\otimes \Omega^1_{X/T})$. This is an absolute Higgs field on $X$ and generically it is a sum of the pullback of the nilpotent absolute Higgs fields $v_0$ and $v_1$ and on the special fiber it is given by a non-zero element in the image of the map
\[H^0(X_s,\Omega^{\otimes 2}_{X_s})\to H^0(X_s,\mathrm{At}(V_X)^\vee \otimes \Omega_{X_s}).\]
\printbibliography
\end{document}